\documentclass[11pt,longbibliography]{article}

\usepackage[utf8]{inputenc}

\usepackage[T1]{fontenc}

\usepackage[a4paper, margin=2.7cm]{geometry}

\usepackage{amsmath, amsthm, amsfonts, amssymb}
\usepackage{mathrsfs}    

\usepackage[colorlinks=true,linkcolor=blue,citecolor=blue,urlcolor=blue,breaklinks]{hyperref}

\usepackage{bbm} 

\usepackage{graphicx,color}
\usepackage{float}
\usepackage{subcaption}

\def\SM#1{\color{red}#1 \color{black}}
\def\SMM#1{\color{magenta}#1 \color{black}}

\def\R{\mathbb{R}}

\def\e{\varepsilon}

\def\z{{\bar{x}}}
\def\z{{\bar{z}}}

\DeclareMathOperator{\supp}{supp}

\def\d{\,\mathrm{d}}

\def\p{\partial}

\newcommand{\f}{\frac}

\def\:{\colon}

\newtheorem{thm}{Theorem}[section]
\newtheorem{cor}[thm]{Corollary}
\newtheorem{lem}[thm]{Lemma}
\newtheorem{prp}[thm]{Proposition}

\theoremstyle{definition}

\theoremstyle{remark}
\newtheorem{rem}[thm]{Remark}
\theoremstyle{example}

\def\thetitle{ A Hamilton-Jacobi approach for the evolutionary dynamics of a model with gene transfer: characterizing monomorphic dynamics for non-concave fitness functions}
\def\theauthor{Alejandro Gárriz and Sepideh Mirrahimi}

\hypersetup{pdftitle={\thetitle}}
\hypersetup{pdfauthor={\theauthor}}

\title{\thetitle}
\author{\theauthor}

\AtEndDocument{\bigskip{\footnotesize
		\noindent
		\textsc{Alejandro Gárriz. Instituto de Matemáticas de Granada \& University of Granada, 18071 Granada, Spain.}\\
		\textit{E-mail address}: \texttt{\href{mailto:alejandro.garriz@ugr.es}{alejandro.garriz@ugr.es}}
		\par
		\addvspace{\medskipamount}
		\noindent
		\textsc{Sepideh Mirrahimi. Institut de Math\'ematiques de Toulouse; UMR5219, Universit\'e de Toulouse; CNRS, Universit\'e Paul Sabatier, F-31077 Toulouse, France.}\\
		\textit{E-mail address}: \texttt{\href{mailto:sepideh.mirrahimi@math.univ-toulouse.fr}{sepideh.mirrahimi@math.univ-toulouse.fr}}
}}

\begin{document}
	
	\maketitle
	
	\begin{abstract}
		We study the asymptotic behaviour of an integro-differential equation describing the evolutionary adaptation of a population structured by a phenotypic trait. The model takes into account mutation, selection, horizontal gene transfer and competition.  Previous works, based on the numerical studies or theoretical study of the corresponding stationary problem, have shown that the dynamics of the solutions are rich and we may expect several qualitative outcomes. In this article, we characterize the dynamics of the solution in two regimes: 1) a situation where the solution concentrates around a dominant trait, evolving gradually to a trait determined by a balance between selection and horizontal gene transfer; 2) a situation where the solution concentrates around a dominant trait which  evolves gradually to a maladapted trait  such that the population becomes extinct (a situation known as the evolutionary suicide).\\
		Our analysis is based on an approach involving Hamilton-Jacobi equations with constraint. Previously, the solutions to such equations were characterized for globally concave growth rates. Here,  we  extend this approach to situations where the growth rate is not globally concave.	
	\end{abstract}
	
	
	\renewcommand{\thefootnote}{\fnsymbol{footnote}} 
	\renewcommand{\thefootnote}{\arabic{footnote}}
	
	\tableofcontents
	\section{Introduction}

	\subsection{Model and question}
	
		The exchange of genetic material between individuals of the same generation is known as \textit{horizontal gene transfer} (HGT), in contrast to vertical gene transfer, which involves the transmission of genetic information from parent to offspring. This genetic information may be altered during birth events due to \textit{mutations}.  Populations are also often subject to \textit{selection} and \textit{competition}. \textit{Selection} corresponds to the fact that the individuals with traits better adapted to their environment will have  a higher reproduction or survival rate, with the consequent effect on the distribution of traits in the population. The death rates of individuals may also be increased due to \textit{competition}   for instance for a limited resource. The combination of the growth rate of an individual or a trait, taking into account \textit{selection} and \textit{competition}, and the interactions between them is called \textit{fitness}. 
	\textit{Fitness} refers to the overall growth rate of an individual or trait, shaped by the processes of \textit{selection} and \textit{competition}, as well as the interactions between these factors, determining the ability to survive and reproduce in a given environment.
	
		We consider the following integro-differential model describing the dynamics of the phenotypic density of 	an asexual population subject to mutation, selection, competition and horizontal transfer:
	\begin{equation}\label{eq:time_dependant_main}
		\begin{cases}
			\partial_t  n( t,  z) = \sigma\partial^2_{zz} n (t,  z)+\left(R(z) -  \kappa \rho( t)\right) n( t, z) +  \tau n( t, z)\displaystyle\int_\R \frac{ n( t,  y)}{ \rho( t)} H(K( z - y))\ {\rm dy}\\
			n( 0,  z) = n_0(z),\\
			n(t, z)>0,\\
			\rho(t)=\displaystyle\int_\R n(t,y)\ {\rm dy}.
		\end{cases}
	\end{equation}
	A variant of this model was derived from a stochastic individual based model in \cite{BilliardColletFerriereMeleardTran}.
	Here, $n(t,z)$ stands for the phenotypic density of a population, with   $t\in \R^+$ and $z\in \R$ corresponding respectively to time and a phenotypic trait. The diffusion term models the mutations which generate phenotypic variability in the population. Individuals grow at rate $R(z)$, also referred to as the selection term, and are regulated by a uniform competition for resources with intensity $\kappa$. The last nonlinear and nonlocal term in the right hand side of the equation corresponds to the horizontal transfer term. More precisely, the positive value $\tau$ denotes the transfer rate and   $H(K(z-y))$ denotes the transfer flux from trait $y$ to trait $z$, with  $K$  a positive steepness parameter for the transfer flux (see Section \ref{sec:assumptions} to better understand its role).
	
	In this article, we will provide an asymptotic analysis of \eqref{eq:time_dependant_main} considering a small mutation rate $\sigma$ and large times. We show that the asymptotic behaviour of the solution can be described via a method involving Hamilton-Jacobi equations. This method has been widely used to study models of evolutionary adaptation considering   mutation, selection and competition \cite{BarlesPerthame2,CalvezLam-2020,DJMP,Lorz-Mirrahimi-Perthame - 2011,Mirrahimi-Roquejoffre - 2016}. We next use this Hamilton-Jacobi formulation of the problem to characterize the dynamics of the solution in two scenarios: (1) a situation where the population converges in long time to an optimal trait, determined by a balance between selection and horizontal gene transfer, (2) a situation where the population converges in long time to unfit traits leading to its extinction. This situation is known as the evolutionary suicide. Our analysis also provides insights on the dynamics of the solution in more complex situations where we expect oscillatory behaviours and emergence in long time of several dominant traits.
	
	A by-product of our analysis is the extension of previous regularity and uniqueness results within the Hamilton-Jacobi method \cite{Lorz-Mirrahimi-Perthame - 2011,Mirrahimi-Roquejoffre - 2016} considering globally concave fitness functions, to the case  with only locally concave fitness functions.  We present such a uniqueness and regularity result considering  fitness functions with a constant local concavity zone,   independently of the rest of the paper in Section~\ref{sec:Dyn_Prog},  since  it can be useful for applications to other models.   The situation in this paper is however more intricate since the local concavity zone of the fitness function evolves with time. In Section \ref{sect:non_concave}, we   show, via the  asymptotic analysis of \eqref{eq:time_dependant_main}, how to apply the method to such a situation. We believe that this method will have many applications in the asymptotic analysis of models from evolutionary biology.
	
	\subsection{Biological motivation and the expected outcomes}

	Bacterial plasmids are a typical example of horizontal transfer mechanisms that influence the genetic distribution of a population, see~\cite{DR.ER.SB:11, FS.BL:77}. Plasmids are small circular doubly stranded DNA, physically separated from the chromosomal DNA, which may be replicated and transferred  from one cell to another, when they are in contact, independently of the chromosome.  They can modify strongly the fitness of their hosts, since they carry factors that can be beneficial for the survival of the bacteria and lead to a selective advantage, as for genes for antibiotic resistance, but they also convey certain fitness costs, like reduced reproduction rate. The question that rises is the following: what is the  outcome of the trade-off between the fitness costs of the plasmids and their advantage by their accelerated spread?
	
	Even though  these exchanges of genetic information, intra and intergenerational, are not mutually exclusive, the models that consider only selection and mutation, where the population is usually driven to the fittest traits, have received more attention in the past, see for example~\cite{BarlesPerthame2,LD.PJ.SM.GR:08,DJMP,Lorz-Mirrahimi-Perthame - 2011}. However, HGT has a major role in the evolution and adaptation of many organisms, as for instance in the evolution of bacterial virulence or antibiotic resistance \cite{HO.JL.EG:00}. 	Recent studies have shown that the interplay between mutation, HGT and selection may lead to new behaviours when compared to models considering only selection and mutation~\cite{SB.PC.RF:16,BilliardColletFerriereMeleardTran,CalvezFigueroaAl,Garriz-Leculier-Mirrahimi}. In \cite{SB.PC.RF:16,BilliardColletFerriereMeleardTran,CalvezFigueroaAl}, three phenomena were observed numerically. First, similarly to the behaviour of most models without horizontal transfer, we have the convergence in long time of the genetic distribution to an optimal trait which, in general, may not coincide with the optimal trait of the selection term of the model, due to the effect of the HGT. Second, the horizontal transfer might drive the population to unfit traits, leading to its extinction, a situation referred to as \textit{evolutionary suicide}. Third, the population can be   driven to  unfit traits leading almost to its extinction but a re-emergence of an apparently extinct trait with positive fitness can ensue; this is known as \textit{evolutionary rescue}. Such behaviour is then repeated in a   cyclic fashion. Note that, for example, in the case where one is studying a population of a pathogen  the third scenario may for instance be interpreted as a re-emergence of antibiotic resistance, while the second one may correspond to a successful treatment where the pathogen is eliminated. A possible fourth outcome of this interplay was characterized more recently in~\cite{Garriz-Leculier-Mirrahimi} via a theoretical study of an integro-differential model.  In this scenario, the steady population distribution concentrates around not one but several traits, depending on the strength of the HGT relative to the selection term. The numerical investigations in~\cite{Garriz-Leculier-Mirrahimi} and the present article, show however that the long time convergence of the solution to such steady distribution may arise via transitory oscillations of the phenotypic distribution close to the cyclic scenarios   numerically observed in \cite{SB.PC.RF:16,BilliardColletFerriereMeleardTran,CalvezFigueroaAl}. This suggests that the third and the fourth scenarios mentioned above may correspond to the same scenarios. The mathematical identification and description of these complex scenarios   has been  the main motivation for the present article.

	\subsection{State of the art}

	Our study stems from a series of works based on stochastic individual based models motivated by the eco-evolutionary dynamics of bacterial plasmids \cite{SB.PC.RF:16,BilliardColletFerriereMeleardTran,CalvezFigueroaAl,NC.SM.CT:21,Garriz-Leculier-Mirrahimi}. More specifically, in \cite{BilliardColletFerriereMeleardTran} a stochastic individual based model was introduced considering a quantitative trait. The authors showed that in the limit of large populations their model converges to an equation close to \eqref{eq:time_dependant_main}, where the mutation term is modelled via an integral kernel rather than a Laplace term. The three types of behaviour discussed earlier (convergence to a certain trait, evolutionary suicide and cyclic evolutionary rescue) were identified in the numerical simulations in \cite{BilliardColletFerriereMeleardTran}. In \cite{CalvezFigueroaAl} similar types of behaviour were observed, when the stochastic simulations were compared with the numerical resolution of an integro-differential model derived in \cite{BilliardColletFerriereMeleardTran}, considering small mutational effects. In \cite{NC.SM.CT:21} a stochastic model with a finite number of strains was studied theoretically. In the particular case of three strains, a periodic behaviour was captured for a certain range of parameters.
	
	In~\cite{Garriz-Leculier-Mirrahimi} the authors investigated the existence and the shape of equilibria of~\eqref{eq:time_dependant_main}. The small diffusion regime was   studied using an approach based on Hamilton-Jacobi equations, proving the convergence, as $\sigma\to 0$, of a logarithmic transformation of solutions of the elliptic version of problem~\eqref{eq:time_dependant_main} to the solution, in the viscosity sense, of a Hamilton-Jacobi equation with a constraint. The  authors proved the existence of equilibria where the population concentrated around one, two or more traits.   Notice that this scenario was not captured in the previous studies discussed earlier.   In this case, the population seems to initially undergo an  oscillatory phase close to the cyclic behaviour observed in \cite{SB.PC.RF:16,BilliardColletFerriereMeleardTran,CalvezFigueroaAl} and to eventually concentrate  around the expected dominant traits. The  equilibria with one or  two concentration points were fully characterized and  several numerical simulations were conducted, depicting convergence of the solution to these complex states, with one, two or more concentration points.

	Less closely related to our work, \cite{LEVIN1979, FS.BL:77} study a model
 considering a finite number of traits  and using ordinary differential equations. Models in a population genetics context without ecological concern were considered in~\cite{AN.GK.EK:05,ST.SB:13}, and finally some integro-differential models of horizontal transfer have been inspected in a different context than our work in \cite{PH.LF.PM.GW:09,PM.GR:15}.

	In this article, we provide an asymptotic analysis of solutions of \eqref{eq:time_dependant_main} considering small mutational effects, i.e. with $\sigma$   small, and long time. Our analysis  provides a partial theoretical description of the behaviours observed numerically in \cite{BilliardColletFerriereMeleardTran,CalvezFigueroaAl,Garriz-Leculier-Mirrahimi}. Note that here the effect of the mutation has been modelled using the Laplace operator, instead of an integral kernel as in  \cite{BilliardColletFerriereMeleardTran,CalvezFigueroaAl}, to reduce the technicality of the analysis, but we believe that this choice would not modify the qualitative behaviour of the solution in the limit of vanishing mutations. To perform our analysis, we use an approach based on Hamilton-Jacobi equations. This approach was first introduced in \cite{DJMP} and then widely developed to study models of quantitative traits from evolutionary biology (see for instance \cite{BarlesMirrahimiPerthame, BarlesPerthame1, BarlesPerthame2,Lorz-Mirrahimi-Perthame - 2011}). A closely related approach was also previously used in the geometric optics approximation of solutions of reaction-diffusion equations (see for instance \cite{ EvansSouganidis, Freidlin}). In the case of models with horizontal gene transfer this approach was used in~\cite{CalvezFigueroaAl} which provides some heuristic computations of the problem and   in~\cite{Garriz-Leculier-Mirrahimi} which studies the steady solutions of \eqref{eq:time_dependant_main}. Here, we treat the time dependent problem. To this end, we extend previous regularity and uniqueness results within the Hamilton-Jacobi method  in~\cite{Lorz-Mirrahimi-Perthame - 2011,Mirrahimi-Roquejoffre - 2016}, which consider some global concavity assumptions of the growth rate, to situations where the growth rate is only locally concave.
	
	\subsection{An adimensional parameterization of the problem}
	\label{sec:change_of_variables}

	We introduce a dimensionless parametrization of the problem via the following  change of variables
	$$
	\tilde t= r\varepsilon t,\quad \tilde z= K z,\quad \tilde n(\tilde t, \tilde z)=\frac{\kappa}{r K}\cdot n\left(\frac{\tilde t}{\varepsilon r},\frac{\tilde z}{K}\right), \quad \varepsilon^2=\frac{\sigma K^2}{r}, \quad \tilde R(\tilde z)=\frac{R(\frac{\tilde z}{K})}{r} ,\quad \tilde \tau = \frac{\tau}{r},
	$$
	where $r$ is defined as
	$$
	r:=\max\limits_{z\in \R} R(z).
	$$
	The problem \eqref{eq:time_dependant_main} is then written (we drop the tildes   for the sake of readability) as
	\begin{equation}\label{eq:main}
		\begin{cases}
			\begin{aligned}
			\varepsilon \partial_t n_\varepsilon(t,z)=&\varepsilon^2 \partial^2_{zz}n_\varepsilon (t,z)+\left(R(z) -  \rho_\varepsilon(t) \right) n_\varepsilon(t,z)\\
			& +  \tau\cdot n_\varepsilon(t, z)\displaystyle\int_\R \frac{ n_\varepsilon(t, y)}{ \rho_\varepsilon(t)}H( z - y)\ {\rm dy}
			\end{aligned}\\
			n_\varepsilon(0,z)=n_{\varepsilon,0}(z)>0,\\
			n_\varepsilon(t, z)>0,\\
			\rho_\varepsilon(t)=\displaystyle\int_\R n_\varepsilon(t,y)\ {\rm dy}.
		\end{cases}
	\end{equation}
	Note that in this new version, the selection term is re-normalized such   that
	$$
	\max_{z\in\R}R(z)=1.
	$$
	In the particular case where $R(z)=r-gz^2$, which we will study in detail later, we also consider the following change of variable
	$$
	\tilde g=\frac{g}{r K^2},
	$$
	which leads, again after dropping the tilde for the sake of readability, to
	\[
		\begin{cases}
		\begin{aligned}
			\varepsilon \partial_t n_\varepsilon(t,z)=&\varepsilon^2 \partial^2_{zz}n_\varepsilon (t,z)+\left(1-gz^2 -  \rho_\varepsilon(t) \right) n_\varepsilon(t,z)\\
			& +  \tau\cdot n_\varepsilon( t,z)\displaystyle\int_\R \frac{ n_\varepsilon(t, y)}{ \rho_\varepsilon(t)}H( z - y)\ {\rm dy},
			\end{aligned}\\
			n_\varepsilon(0,z)=n_{\varepsilon,0}(z)>0,\\
			n_\varepsilon(t, z)>0,\\
			\rho_\varepsilon=\displaystyle\int_\R n_\varepsilon(t,y)\ {\rm dy}.
		\end{cases}
	\]

	\subsection{Assumptions}
	\label{sec:assumptions}

	We provide the first set of assumptions on the growth term $R(z)$. We will assume that $R\in C^2(\R)$ and that there exists a set of positive constants $K_1,..., K_5$ and $\underline{K}_0, \overline{K}_0$ such that
\begin{equation}\label{hyp:R}\tag{HR1}
\begin{aligned}
	K_3-K_4z^2&\leq R(z)\leq K_1-K_2z^2,\\
	-\underline{K}_0&\leq R''(z)\leq -\overline{K}_0,\\
	&\quad\|R'''\|_{L^\infty(\R)}\leq K_5.
\end{aligned}	
\end{equation}

We define
$$
D_R=\{z\in \R\,\ :\ R(z)>0\}.
$$
From Assumption \eqref{hyp:R} we deduce that $D_R$ is a bounded open interval. Let us also define $z_\mu\in \R$ as the unique point such that
$$
\tau+R'(z_\mu)=0.
$$

For some of our results we also  suppose that 
\begin{equation}\label{hyp:R2}\tag{HR2}
	 	z_\mu\in D_R.
\end{equation}
A typical example of such growth term $R$ is given by
	\begin{equation}\label{R:quadratic}
		R(z)=1-gz^2,\quad \text{with} \quad \tau^2\leq 4g,
	\end{equation}
which, as commented, will be studied in detail later on in the article.	
	
	We make the following assumptions on the transfer term $H$:
	\begin{equation}\label{eq:hypothesis_H}\tag{HT}
		\begin{aligned}
			(1)\quad &H \in C^3(\mathbb{R}) \text{, with bounded derivatives, is odd and monotone increasing from -1 to 1}.\\
			(2)\quad &H(0)=0, H'(0)=1\text{ and } H''(z)<0\text{ for all }z>0 .\\
			(3)\quad &\text{There exists a positive }z_H\text{ such that for all }|z|>z_H, H'''(z)>0,\\
			&\text{while for all }|z|\leq z_H, H'''(z)\leq 0.
		\end{aligned}
	\end{equation}
	Two examples of functions satisfying these assumptions are
	$$
	H(z)=\tanh(z)\qquad \text{or}\qquad H(z)=\frac{2}{\pi}\arctan(z).
	$$
	One can think of this kernel as $H(z-y) = \alpha(z-y) - \alpha(y-z)$ with $\alpha$ a smooth function that behaves like a Heaviside step function. Then, one would consider that the transfer arises only from larger traits $y$ to smaller traits $z$, with $z<y$ and the transfer rate between $y$ to $z$ would be given by $\alpha(z-y)$. This choice of transfer term is motivated by the example of plasmids which are transmitted from one bacterium to another by cell-to-cell contact. 
	Next, the value $\tau$, which is understood as the strength of the transfer, is considered to be strictly positive, i.e.,
	$$
	\tau>0.
	$$

	We next provide the conditions on the initial datum. We consider an initial datum $n_{\varepsilon,0}(z)>0$, a continuous function such that
	$$
	u_{\varepsilon,0}(z):=\varepsilon\ln (n_{\varepsilon,0}(z))
	$$
	satisfies, for all $\varepsilon\in(0,\varepsilon_0)$ for a certain $\varepsilon_0>0$, for a set of positive constants $A_1$, $A_2$, $B_1$, $B_2$, $C_1$, $C_2$ and $C_3$,  	
	\begin{equation}\label{eq:hypothesis_u_0}\tag{H0}
	\begin{aligned}
	-A_1-B_1z^2\leq  u_{\varepsilon,0}(z)&\leq A_2-B_2z^2,\\
	-C_2\leq\partial^2_{zz}u_{\varepsilon,0}(z)&\leq-C_1,\\
	\max\limits_{z\in \R} u_{\varepsilon,0}(z)&=u_{\varepsilon, 0}(z_{\varepsilon,0})\text{ for a unique }z_{\varepsilon,0}\in D_R,\\
	\|\partial^3_{zzz}u_{\varepsilon,0}\|_{L^\infty(\R)}&\leq C_3, \text{ and}\\
	\lim\limits_{\varepsilon\to 0} u_{\varepsilon,0}(z) &= u_0(z)\text{ (locally uniformly),}\\
	\text{with }  \max_{z\in \R} u_0(z)&=u_0(z_0)=0
	\end{aligned}
	\end{equation}
	for a certain $z_0\in D_R$.	A classical example of such an initial datum is given by
	\[
	n_{\varepsilon,0}(z)=\frac{1}{\sqrt{\varepsilon}}e^{-c\frac{(z-z_0)^2}{\varepsilon}},\text{ meaning that }\lim\limits_{\varepsilon\to 0} u_{\varepsilon,0}(z) = -c(z-z_0)^2,
	\]
	for any $c>0$. The condition   $z_{\varepsilon,0}\in D_R$ ensures that the initial population is not maladapted. It implies indeed that the phenotypic distribution is initially concentrated around a trait that has a positive growth rate. 
	
	Finally, we assume that the initial population size $\rho_\varepsilon$ satisfies, for two positive constants $\rho_m,\rho_M$,
	\begin{equation}\label{eq:hypothesis_rho}\tag{HM}	
	\begin{cases}
	0<\rho_m\leq\rho_\varepsilon(0)\leq \rho_M\quad\text{for all }\varepsilon\in (0,1),\\
	\rho_\e(0)\to \rho_0=R(z_0)>0,\quad \text{as $\e\to 0$}.
	\end{cases}
	\end{equation}

	\subsection{Main results}
	\label{subsec:results}
	
  We expect the solution $n_\e$ to concentrate, as  $\e\to 0$,  around certain dominant traits, forming Dirac's delta functions in the limit. In order to identify such singular limits, we use an approach involving Hamilton-Jacobi equations \cite{BarlesMirrahimiPerthame, BarlesPerthame2,DJMP}. The main ingredient in this approach is to perform the following Hopf-Cole transformation
  \begin{equation}\label{eq:transformed_hopf_cole}
	u_\varepsilon(t,z):=\varepsilon\cdot \ln\big(n_\varepsilon(t,z)\big),
\end{equation}
 which allows us to unfold the singularity of the problem. Indeed, while $n_\varepsilon$ tends, as $\varepsilon\to 0$, to a singular measure, $u_\varepsilon$ converges to a continuous function $u$ which solves a Hamilton-Jacobi equation. The main idea is to first study the limit of $u_\varepsilon$, which we call $u$, and next to use some information on the function $u$ to identify $n$. 
 
Replacing~\eqref{eq:transformed_hopf_cole}    in~\eqref{eq:main} we obtain the following equation on $u_\varepsilon$
\[
	\begin{cases}
		\partial_t u_\varepsilon(t,z)=\varepsilon \partial^2_{zz}u_\varepsilon (t,z)+(\partial_{z}  u_\varepsilon(t,z))^2+R(z) -  \rho_\varepsilon(t) + \Phi_\varepsilon(t,z) ,\\
		u_\varepsilon(0,z)=u_{\varepsilon,0}(z),\\
		\rho_\varepsilon(t)=\displaystyle\int_\R n_\varepsilon(t,y)\ {\rm dy}.
	\end{cases}
\]
{where
\[
\Phi_\varepsilon(t,z):= \tau\displaystyle\int_\R \frac{ n_\varepsilon(t, y)}{ \rho_\varepsilon(t)}H( z - y)\ {\rm dy}.
\]}
Then, passing formally to the limit in the equation on $u_\e$, we obtain 
\begin{equation}\label{eq:viscosity_u}
		\begin{cases}
			\partial_t u(t,z)= | \partial_z u(t,z)|^2 + R(z) - \rho(t) +\Phi(t,z),\quad & t>0, z\in\R,\\
			u(0,z)=u_0(z), &z\in \R.
		\end{cases}
\end{equation}
 with,
\[
\Phi(t,z):=\tau \int_\R \frac{n(t,y)}{\rho(t)}H(z-y)\ {\rm d}y, \quad \text{ and } \quad \rho(t) = \int_{\mathbb{R}} n(t,y)dy.
\]
  
   Let
  \[
  \rho_{max}:=\max\{\max\limits_{z\in \R} R(z), \rho_M\}.
  \]
  We prove the following.
  
\begin{thm}\label{thm:limit_epsilon}
		Let conditions~\eqref{hyp:R},~\eqref{eq:hypothesis_H},~\eqref{eq:hypothesis_u_0} and~\eqref{eq:hypothesis_rho} be satisfied. As $\varepsilon\to 0$ and along subsequences,
		 $n_\e$ converges in $L^\infty(w \ast(\R^+);\mathcal M^1(\R))$ to a measure $n\in L^\infty(\R^+;\mathcal M^1(\R))$, $\rho_\e$ converges in $L^\infty(w \ast(\R^+))$ to a function $\rho\in L^\infty(\R^+)$ and $\Phi_\e$ converges in  $L^\infty(w\ast(\R^+); C^2 (\R))$ to a function $\Phi\in L^\infty(\R^+;C^2( \R))$ with
\begin{equation}
\label{bound-rho}
\rho(t) = \int_{\mathbb{R}} n(t,y)dy,\qquad	   0\leq \rho(t) \leq \rho_{max},	\qquad\text{and}\quad   |\Phi(t,z) |\leq \tau.
\end{equation}
 The solutions $u_{\varepsilon }$ converge locally uniformly to a continuous function $u$ that is a viscosity solution of the Hamilton-Jacobi equation~\eqref{eq:viscosity_u}. Moreover the solution $u$ is given by the Dynamic Programming Principle
\begin{equation}\label{eq:dyn_prog}
	u(t,z) = \sup_{(\gamma(s),s)\in\R\times[0,t], \gamma(t)=z}\Big\{f_t(\gamma)\ ,\text{ with }\gamma\in W^{1,2}([0,t] : \R)\Big\},
\end{equation}
with
\[
	f_t(\gamma)=u_0(\gamma(0))+\int_0^t\Big(-\frac{|\dot\gamma(s)|^2}4+R(\gamma(s))-\rho(s)+\Phi(s,\gamma(s))\Big)ds,
\]
 We also have
 \begin{equation}\label{eq:support}
 u(t,z)\leq 0\quad\text{and}\quad\supp(n(t,\cdot))\subseteq \{z\in \R\ :\ u(t,z)=0\}.
 \end{equation}

\end{thm}	

Note that for any $t_0\in [0,t)$, we also have
\[
	u(t,z) = \sup_{(\gamma(s),s)\in\R\times[t_0,t], \gamma(t)=z}\Big\{f_t(\gamma)\ ,\text{ with }\gamma\in W^{1,2}([t_0,t] : \R)\Big\},
\]
with
\[
f_t(\gamma)=u(t_0,\gamma(t_0))+\int_{t_0}^t\Big(-\frac{|\dot\gamma(s)|^2}4+R(\gamma(s))-\rho(s)+\Phi(s,\gamma(s))\Big)ds.
\]

Notice that if the set of zeroes of $u$ is a set of discrete points, then $n$ is a sum of Dirac's delta functions supported in such a set.	The focus of this article is, in particular, on the so-called \textit{monomorphic} situation, where there is a unique function $\z(t)$ such that,  for a.e. $t\geq 0$,
	\[
	\z(t)=\supp(n(t,\cdot))=\{z\in \R\ :\ u(t,z)=0\}\quad\text{and}\quad u(t,z)<0\text{ for all }z\neq\z(t).
	\]
	In this article, we will say that $n$ or $u$ are \textit{continuous monomorphic} {iff} $\z(t)$ is continuous for all $t\geq 0$ and,   $n$ and $u$ will be called \textit{discontinuous monomorphic} {iff} there exists a $t_1>0$ such that
 $\z(t)$ is discontinuous at $t_1$.
Consequently, we will say that $n$ and $u$ are \textit{continuous monomorphic on $[t_i, t_j]$} {iff} $\z(t)$ is continuous for all $t\in[t_i, t_j]$.

We next focus on the specific growth rate $R(z)=1-gz^2$ and define
	\begin{equation}\label{eq:definition_mu}
	\mu=\frac{\tau}{2g}.
	\end{equation}
	This quantity will determine the behaviour of $n$ in the following sense.
	We expect that there is a positive constant $\mu_1$ such that if $\mu\leq \mu_1$ then the limit solution $n$ remains continuous  monomorphic. This threshold corresponds indeed to the one below which the solution of the elliptic version of~\eqref{eq:main} is monomorphic, see~\cite{Garriz-Leculier-Mirrahimi}.
	
 	We introduce the following function 
 	\[
 	F(t,z)  = R(z) - \rho(t) + \Phi(t,z),
 	\]
   	which corresponds to the right-hand side of \eqref{eq:viscosity_u} minus the gradient term. We will refer to $F(t,z)$ as the fitness function. Briefly, if we consider a monomorphic regime (implying that $\Phi (t,z)=\tau H(z-\bar{z}(t))$), then we can prove that $F(t,\z)$ converges, as $t$ grows, to the function
	\[
	F_\mu(z)=-g(z^2-\mu^2)+\tau H(z-\mu).
	\]
	We define then $\mu_1$ as 
	\[
	\mu_1:=\sup\{\mu>0\ :\ F_\mu(z)<0\text{ for all }z\neq \mu, F_\mu(\mu)=0\}.
	\]
	It is proved in~\cite[Section 4.2]{Garriz-Leculier-Mirrahimi} that this value can be computed by finding the unique positive value $d_1$ such that
	\begin{equation}\label{eq:monomorphic_range}
	2H(d_1)=d_1(1+H'(d_1)),
	\end{equation}
	which is well defined thanks to~\eqref{eq:hypothesis_H},  and then we have that
	\[
	\mu_1=\frac{d_1}{1-H'(d_1)}.
	\]
	When $\mu=\mu_1$, the function $F_{\mu_1}$ is still non-positive but it has two zeroes, one at $z=\mu_1$ and the other at $\mu_1-d_1$. Finally, if $\mu>\mu_1$, then $F_{\mu}(z)>0$ for some $z<\mu$.
	
	 It was proved in~\cite{Garriz-Leculier-Mirrahimi} that when $\mu \in [0,\mu_1]$, then the stationary solutions of~\eqref{eq:main} are monomorphic. Here, we prove that in this regime the solution of the time dependent problem is continuous monomorphic for a wide range on initial data. This dominant trait evolves and converges, as time grows to infinity, to the unique monomorphic stationary solution of~\eqref{eq:main}. We recall the definition of $z_0$ in \eqref{eq:hypothesis_u_0} as the point where the initial datum $u_0$ attains its maximum. Note that under the particular choice of $R(z)=1-gz^2$ hypothesis~\eqref{hyp:R} is satisfied automatically.

 \begin{thm}\label{thm:limit_epsilon_non_concave}
		Assume~\eqref{eq:hypothesis_H},~\eqref{eq:hypothesis_u_0} and~\eqref{eq:hypothesis_rho}, $z_0\leq \mu$ and let $R(z)=1-gz^2$.  \\
		(i) Assume that  $0\leq \mu\leq \mu_1$ and $\tau \leq 2\sqrt{g}$. Then, there exists an open set $\mathcal C\subset \R^+\times \R$, such that $u_\varepsilon$ converges, as $\varepsilon\to 0$,  to a continuous function $u$, with $u|_\mathcal C \in L^\infty_{t,{\rm loc}}  W^{3,\infty}_{x,{\rm loc}}( \mathcal C) \cap C^1(\mathcal C)$, which is the unique viscosity  solution of 
		\begin{equation}
		\label{eq:z,u}
		\begin{cases}
		\p_t u(t,z)=|\p_z u(t,z)|^2+ g\z^2(t)- gz^2 +\tau H(z-\bar z(t)),&\text{in $\R^+\times \R$} \\
		\bar{z}'(t)=(-\partial_{zz}u(t, \bar{z}(t)))^{-1}\,(\tau-2g\bar{z}(t)),&\text{in $\R^+$},
		\end{cases}
		\end{equation}
		with $(t,\z(t))\in \mathcal{C}$.
		 The function  $u(t,\cdot)$ has a unique maximum point at $\z\in C^1(\R^+)$, with $u(t,\z(t))=0$ and $\bar{z}(t)\to \mu$ as $t\to \infty$. Moreover, $u$ is strictly concave with respect to $z$ in $\mathcal C$.  \\
	Furthermore, as $\varepsilon\to 0$ and along the same subsequences as in  Theorem~\ref{thm:limit_epsilon}, $\rho_\varepsilon$ and $\Phi_\varepsilon$ converge pointwise to
	\begin{equation}
	\label{rho}
		\rho(t)= 1-g\bar{z}(t)^2>0,\quad \Phi (t,z)=\tau H(z-\bar{z}(t)),
	\end{equation}
	and $n_\varepsilon$ converges weakly in the sense of measures to
	\[
	n(t,z):=\rho(t)\cdot \delta_{\bar{z}(t)}(z),
	\]
	implying that $n$ is continuous monomorphic.\\
	(ii) Assume that  $0\leq \mu\leq \mu_1$ and $\tau > 2\sqrt{g}$. Then, the statements of the previous point hold true   until a time  $T_{\rho}$ at which $\rho(T_\rho)=0$, $\bar z(T_\rho)=1/\sqrt{g}$. 
 	
	\end{thm}

\begin{rem}[The globally concave case]
	In the particular case where $R(z)$ is such that hypothesis~\eqref{hyp:R} and~\eqref{hyp:R2} are satisfied and the fitness function $F(t,z)$ is globally uniformly and strictly concave, which would hold for $\tau$ small enough,  then we are in the framework of~\cite{Mirrahimi-Roquejoffre - 2016} and we obtain almost immediately the same results of Theorem~\ref{thm:limit_epsilon_non_concave} under the same hypothesis plus the concavity condition. 
\end{rem}

\begin{rem}
	In Section~\ref{sect:z_0>mu} we present a result similar to Theorem~\ref{thm:limit_epsilon_non_concave} regarding the case $z_0>\mu$, which is derived analogously to the case $z_0\leq \mu$ under an extra hypothesis over $F(0,z)$, where we assume that $F(0,z)$ may take positive values only in   a set of the form $(z_0, z_0+\delta)$ for some $\delta>0$. The set of positive values of $F(0,z)$ may indeed contain other points for some values of $z_0>\mu$, in which case our result may not hold anymore.
\end{rem}

\begin{rem}
Notice that the theorem above implies that when $\tau\geq2\sqrt{g}$, the population gets extinct, in finite time if $\tau>2\sqrt{g}$, and as $t\to\infty$ if $\tau=2\sqrt{g}$. In this case, the horizontal transfer drives the population to unfit traits, leading to its extinction, a situation referred to as \textit{evolutionary suicide}.
\end{rem}

Several difficulties arise in the proof of Theorem \ref{thm:limit_epsilon_non_concave}. First, we lack time regularity estimates on $\rho_\e$ and $\Phi_\e(t,z)$ to obtain their pointwise convergence as $\e\to 0$. Moreover, in order to derive \SM{\eqref{eq:z,u}}, we need regularity and concavity estimates on $u$. Such estimates were previously proved in \cite{Mirrahimi-Roquejoffre - 2016} in a situation where the fitness function is globally concave, which is not the case here. In order to overcome these difficulties, we extend the method introduced in \cite{Mirrahimi-Roquejoffre - 2016} to situations where the fitness function is only locally concave, with its local concavity zone possibly evolving with time. In this way, we obtain the required regularity and concavity estimate on $u$ which we next use   to obtain the pointwise convergence of $\rho_\e$ and $\Phi_\e$.

Notice that the fitness function $F(t,z)$, when considering  $R(z)=1-gz^2$ and a monomorphic density $n(t,z)=\rho(t)\delta (z-\bar z(t))$ at the limit $\e\to 0$,  is given by
\[
F(t,z)=1-gz^2-\rho(t)+\tau H(z-\z (t)).
\]
Moreover, we expect that $F(t,\z(t))=0$ and hence, similarly  to \eqref{rho},   we have $\rho(t)=1-g\z(t)^2$. We deduce that
\[
F(t,z)=-gz^2+g\z(t)^2+\tau H(z-\z(t)). 
\]
Therefore, $F(t,z)$ can be written as follows 
\begin{equation}
\label{F-G}
F(t,z)=G(z,\z(t)),\qquad G(z,y)=-gz^2+gy^2+\tau H(z-y).
\end{equation}
Notice that $G$ is a continuous function. The possible discontinuity in time of $F$ would come from the discontinuity of $\z(t)$.
Moreover, one can verify that 
\[
G(y,y)=0,\qquad \text{for all $y\in \R$.}
\]
In many  models from evolutionary biology, the fitness function can be written in this form (see e.g. \cite{SM.BP.PS:12,Mirrahimi-Roquejoffre - 2016,SF.SM:18}).

Indeed, as a product of our analysis, we obtain in Section~\ref{sec:Dyn_Prog} some new results of independent interest regarding the existence, uniqueness and regularity properties, of a family of Hamilton-Jacobi equations posed in $\R^d$ of the form
\begin{equation}\label{eq:G}
	\begin{cases}
		\partial_t v(t,z) = |D v(t,z)|^2+G(z,\z(t)),&   t\in [0, T],z\in\R^d,
		\\
		\max_{z\in\R}v(t,z)=v(\z(t),t)=0,& t\in [0, T],
		\\
		v(0,z)=v_0(z),& z\in\R^d,\\
		v_0(z_0)=0 &\text{for some }z_0\in \R,\quad v_0(z)<0 \text{ for all }z\neq z_0,
	\end{cases}
\end{equation}
where, among other conditions, $G$ satisfies a concavity estimate in a local sense; namely that there exists a domain $\Omega_0$ and a couple of constants $\underline{K}_1>\overline{K}_1>0$ such that for all $t\in[0,T]$ and $w\in\R^d$, 
\[
- 2\underline{K}_1|w|^2 \leq w\ D^2 G(z,\z)\ w^T \leq - 2\overline{K}_1|w|^2 < 0, \quad \text{for }z \in \overline \Omega_0,\ \z \in \overline \Omega_0.
\]
Here and in the rest of the article $Dv$ stands for the gradient of $v$ and $D^2G$ for the hessian matrix of $G$, all with respect to the $z$ variable.

The situation in this article is however more complex than the case above since the local concavity zone of the fitness function evolves over time. We show in Section~\ref{sect:non_concave} how to deal with this more complex situation. In Section~\ref{sect:beyond_monomorphism} we additionally explore insights beyond the monomorphic scenario.

 {\subsection{To what extent can the assumptions be relaxed?}}

			Let us discuss briefly the relevance of the hypotheses in~\eqref{eq:hypothesis_H} and~\eqref{hyp:R} in order to provide some insight into  a more general theory. The regularity required for $R$ and $H$ is essential in our analysis, see Sections~\ref{sec:Dyn_Prog} and~\ref{sect:non_concave}. Less regularity is needed if we are only interested in Theorem~\ref{thm:limit_epsilon}.

			Concerning the function $R$, the assumptions in~\eqref{hyp:R} are classical but not fundamental in our study. The first hypothesis could be replaced by other forms of local control, while the second assumption is there to ensure that the fitness function $F(t,z)$,  defined  in Section \ref{subsec:results},   is at least locally concave in a certain set, again see Sections~\ref{sec:Dyn_Prog} and~\ref{sect:non_concave}. Both these assumptions can indeed be weakened, as long as we ensure that the fitness function satisfies the conditions of Section~\ref{sec:Dyn_Prog}.   {The specific choice of $R$ in~\eqref{R:quadratic}, that we impose for Theorem~\ref{thm:limit_epsilon_non_concave},  together with the choice of $H$,   is however fundamental in order to properly define the quantity $\mu$ that determines the long time} behaviour of the solution, see~\eqref{eq:definition_mu} and the subsequent explanation.

		The oddness of $H$ (which implies $H(0)=0$) and its monotonicity property are not only biologically relevant but fundamental in many of our computations. The rest of the hypotheses in~\eqref{eq:hypothesis_H} are there, again, to provide the local concavity of the fitness function, but they also play other roles.   {In particular, they determine, alongside the choice $R(z)=1-gz^2$, the value of  $\mu$ in~\eqref{eq:definition_mu} (which would be replaced by $\mu=H'(0)\tau/2g$ if we remove the condition $H'(0)=1$)  and they ensure} {that the equation that determines the monomorphic range given in~\eqref{eq:monomorphic_range} has a unique positive solution; see the explanations before Theorem~\ref{thm:limit_epsilon_non_concave} and~\cite[Section 4.2]{Garriz-Leculier-Mirrahimi}. As long as we can ensure the local concavity of the fitness function and the existence and uniqueness of the positive solution of the equation determining the monomorphic range, these assumptions can also be weakened.

\subsection{Organization of the article}

The article is organised as follows. In Section~\ref{sec:Dyn_Prog} we present some useful results for Hamilton-Jacobi equations in $\R^d$ of the {form~\eqref{eq:G}} that we will use later in the one-dimensional case. Section~\ref{section:limit_varepsilon} is devoted to proving Theorem~\ref{thm:limit_epsilon}, while in Section~\ref{sect:non_concave} we will prove Theorem~\ref{thm:limit_epsilon_non_concave}. In Section~\ref{sect:z_0>mu} we comment on the case $\z(0)>\mu$, which is not covered by Theorem~\ref{thm:limit_epsilon_non_concave}, presenting a partial result when the fitness function satisfies a certain condition. Finally, Section~\ref{sect:beyond_monomorphism} offers some insight in the case $\mu>\mu_1$, where monomorphism is lost, even though this case lies beyond the scope of this article.

\section{Uniqueness and regularity for the Hamilton-Jacobi equation with constraint under a local concavity assumption}\label{sec:Dyn_Prog}
 	
Since the theory that we develop in this section is interesting on its own for the study of Hamilton-Jacobi equations with constraint, we will consider the more general problem \eqref{eq:G}.  It has been proven, under global concavity assumption on $G$ and $v_0$, that this problem can be reduced to a non standard PDE-ODE system, see~\cite{Mirrahimi-Roquejoffre - 2016}. Our objective is to extend this result to a more general framework, where we only assume local concavity on $G$. Notice that, with respect to the problem that we have in mind, we are this time working in dimension $d\geq 1$ and in a possibly bounded time interval. It is also important to highlight here that in Section~\ref{sect:non_concave} we will see how the ideas from this section can even apply to the case where the fitness function is not positive and locally concave in the same set for all times. Indeed, in our model, this set of concavity will change continuously over time.\\

Focusing on the propagation of the local concavity and regularity properties, we will first  study the following  problem
\begin{equation}\label{eq:1}
	\begin{cases}
		\partial_t v(t,z) = |D v(t,z)|^2+F(t,z),&   t\in [0, T],z\in\R^d,
		\\
		\max_{z\in\R}v(t,z)=0,& t\in [0, T],
		\\
		v(0,z)=v_0(z),& z\in\R^d,\\
		v_0(z_0)=0 &\text{for some }z_0\in \R,\quad v_0(z)<0 \text{ for all }z\neq z_0,
	\end{cases}
\end{equation}
 with $z_0$ a real constant,  $v_0$ a nonpositive continuous function and $0<T\leq \infty$. We consider $v$ a solution in the viscosity sense to~\eqref{eq:1}.\\

\noindent\textbf{Assumptions on $F$ and on the initial condition $v_0$}

We  assume that
\begin{equation}\label{eq:regularity_F}
F\in L^\infty_{loc}([0, T]\times \R^d),\quad DF\in C([0,T]\times \R^d)
\end{equation}
and that there exists a non-empty, open,  convex and smooth set $\Omega_0 \subset \R^d $ such that, for all $t\in[0, T]$,
\begin{equation}
\label{hyp1}
- 2\underline{K}_1|x|^2 \leq x\ D^2 F(t,z)\ x^T \leq - 2\overline{K}_1|x|^2 < 0, \quad \text{for any }z \in \overline \Omega_0,\text{ for all }x\in\R^d,
\end{equation}
\begin{equation}
\label{hyp2}
\begin{array}{c}
\forall x\in\R\setminus\Omega_0,\forall y\in  \Omega_0 \qquad F(t,x)<F(t,y),
\end{array}
\end{equation}
\begin{equation}
\label{as:nablaR0}
D F(t,z) \neq 0, \qquad \text{for all }z\in \p \Omega_0,
\end{equation}
and
\begin{equation}\label{asRD3}
\|D^3F(t,\cdot) \|_{L^\infty(\Omega_0)}\leq K_4.
\end{equation}
We make the following assumptions on the initial condition
\begin{equation}
 \label{eq:hypz0}
z_0\in \Omega_0,
\end{equation}
\begin{equation}
 \label{eq:hyp3}
 \forall x\in\R\setminus\Omega_0,\forall y\in \Omega_0, \quad v_0(x) <v_0(y),
\end{equation}
\begin{equation}
\label{as:u0concave}
- 2\underline{L}_1|x|^2 \leq x\ D^2 v_0(z)\ x^T \leq - 2\overline{L}_1|x|^2 < 0, \quad \text{for any }z \in \overline \Omega_0,\text{ for all }x\in\R^d,
\end{equation}
\begin{equation} \label{asuD3}
\|D^3v_0 \|_{L^\infty(\overline \Omega_0)} \leq L_3.
\end{equation}
%

We next state the main results in this framework.

\begin{thm}
\label{thm:section_Dyn_Prog_1}
Let $F$ satisfy~\eqref{eq:regularity_F}--\eqref{asuD3}. Then any viscosity solution $v$ to~\eqref{eq:1} is indeed  classical and strictly concave in the set $\Omega_0$, and we have
\[
\begin{aligned}
&v|_{\Omega_0} \in L^\infty_{\rm loc} ([0,T] : W^{3,\infty}_{loc}( \Omega_0)) \cap C^1([0,T]\times\Omega_0),\\
&D v\in C^1([0,T]\times\Omega_0) \quad \text{and} \quad v(t,\cdot)|_{ \Omega_0}\in C^2( \Omega_0).
\end{aligned}
\]
Furthermore, for all $t\in [0, T]$, the maximum of $v(t,\cdot)$ in $\R^d$ is attained at a single point $\z(t) \in \overline \Omega_0$, a function of time satisfying $\z\in C^1([0, T])$.  
\end{thm}

Note that we do not claim so far that solutions to problem~\eqref{eq:1} are unique. In order to establish uniqueness we need to ask something else to the zero-order term $F$; in our case, that $F$ depends on time only through the trajectory of the curve $\z(t)$, i.e., $F(t,z)=G(z,\z(t))$ as in \eqref{F-G}.

	
We assume that $G:\R^d\times \R^d\to \R$ is continuous function which satisfies 
\begin{equation}
\label{as:Gzz}
G(z,z)=0,\qquad \forall z\in \R^d.
\end{equation}
Translating the hypothesis on $F$ onto $G$, we need that for all $\z \in \overline \Omega_0$,
\begin{equation}\label{eq:assumption-thm-2.2}
	- 2\underline{K}_1|x|^2 \leq x\cdot D^2 G(z,\z)\cdot x^T \leq - 2\overline{K}_1|x|^2 < 0, \quad \text{ for all }z\in \Omega_0 \text{ and } x\in\R^d,
\end{equation}
\begin{equation}
	\label{hyp2_G}
	\begin{array}{c}
		\forall x\in\R\setminus\Omega_0,\forall y\in  \Omega_0 \qquad G(x,\z)<G(y,\z),
	\end{array}
\end{equation}
\begin{equation}
	\label{as:nablaR0_G}
	D G(z,\z) \neq 0, \qquad \text{for all }z\in \p \Omega_0, 
\end{equation}
\begin{equation}
	\label{asRD23_G}
	\left| \f{\p G}{\p \z_i  }(z,\z)\right|+ \left| \f{\p^2 G}{\p \z_i \p z_j}(z,\z)\right| + \left| \f{\p^3 G}{\p \z_i \p z_{j} \p z_k}(z,\z)\right| \leq K_3, \ \text{for  $z\in \overline \Omega_0 $ and $i,\,j,\,k=1,\cdots,d$},
\end{equation}
and
\begin{equation}\label{asRD3_G}
	\|D^3G(\cdot, \z) \|_{L^\infty(\Omega_0)}\leq K_4,
\end{equation}
The symbols $D, D^2$ and $D^3$ refer to derivatives with respect to the variable $z$. Similarly, notice that when we write $D G(\z(t),\z(t))$ we refer only to the first set of variables, $z\in \R^d$; i.e.,
\[
D G\big(\z(t),\z(t)\big) = D G\big(z,\z(t)\big)\big|_{z=\z(t)}.
\]

 \begin{thm}\label{thm:section_Dyn_Prog_2}
Let $G$ satisfy~\eqref{as:Gzz}--\eqref{asRD3_G}.  Solving the constrained problem~\eqref{eq:1} is equivalent to solving the following ODE-PDE system
  \begin{equation}
\label{eq:3}
\begin{cases}
	\partial_t v(t,z) = |D v(t,z)|^2+G\big(z, \z(t)\big),& t\in [0, T],z\in\R^d,\\
	\dot{\z}(t) = \big(-D^2v(t,\z(t))\big)^{-1}\  D G\big(\z(t),\z(t)\big),& t\in [0, T],\\
	v(0, z)=v_0(z),& z\in\R^d,\\
	\z(0)=z_0\in \Omega_0,
 \end{cases}
\end{equation}
with initial conditions satisfying
\begin{equation}\label{eq:initial_cond_v_G}
\max_{z\in\R} v_0(z) = v_0(z_{0}) = 0.
\end{equation}
Consequently, there exists a unique viscosity solution $v$ to~\eqref{eq:1}, with $F(t,z)=G(z,\z(t))$. This solution satisfies
\[
\max\limits_{z\in \R}v(t,z)=v(t,\z(t))=0,
\]
where $\z\in\overline\Omega_0$ is the unique maximum point of $v$ and, moreover, the couple of solutions $(v, \z)$ satisfy
\[
\begin{aligned}
	&(v|_{\Omega_0},\z) \in \big(L^\infty_{\rm loc} ([0,T] : W^{3,\infty}_{loc}( \Omega_0)) \cap C^1([0,T]\times\Omega_0)\big)\times C^1([0,T]),\\
	&D v\in C^1([0,T]\times\Omega_0) \quad \text{and} \quad v(t,\cdot)|_{ \Omega_0}\in C^2( \Omega_0).
\end{aligned}
\]
 \end{thm}

Note that since $G$ is not globally concave, $v$ is not globally concave and hence we cannot obtain a global regularity result for $v$. In particular, $(D^2 v)^{-1}$ is not defined everywhere and  such differential system cannot be well-defined in general. However, here
we prove that starting with an initial data such that $z_0\in \Omega_0$, and under the above assumptions on $G$, the above differential system will be well-defined for all $0\leq t\leq T$. A key argument is that any optimal trajectory $\gamma(\cdot):[0,T]\to \R$ appearing in the representation formula of the Hamilton-Jacobi equation \eqref{eq:1} such that  $\gamma(T)\in \Omega_0$, has always been in the set $\Omega_0$, that is $\gamma(t)\in \Omega_0$ for all $t\in[0,T]$.

\subsection{Proof of Theorem \ref{thm:section_Dyn_Prog_1}}\label{sec:Dyn_Prog_1}

The proof of theorem   \ref{thm:section_Dyn_Prog_1} is closely related to the method introduced in \cite{Mirrahimi-Roquejoffre - 2016}. However, here we have to deal with the difficulty that the fitness function is not globally concave.

Fix   $t\in [0, T]$ and $z\in \overline \Omega_0$. Then the value of the viscosity solution $v$ of the Hamilton-Jacobi equation in \eqref{eq:1} at point $(t,z)$ is given  by the dynamic programming principle: 
$$
v(t,z) = \sup_{(\gamma(s),s)\in\R^d\times[0,t], \gamma(t)=z}\Big\{f_t(\gamma),\text{ where }\gamma\in W^{1,2}([0,t] : \R^d)\Big\},
$$
with
$$
f_t(\gamma)=v_0(\gamma(0))+\int_0^t\Big(-\frac{|\dot\gamma(s)|^2}4+F(s,\gamma(s))\Big)ds.
$$
A crucial element that allows us to adapt the proof from \cite{Mirrahimi-Roquejoffre - 2016} is that the optimal trajectories with ending points (at time $t$) in the set $\Omega_0$ have been in the set $\Omega_0$ for all $s\in [0,t]$.

\textbf{Step 1. Existence  of an optimal trajectory.}

 Let us take a sequence $(\gamma_n)_{n\geq1}$, with $\gamma_n\in W^{1,2}([0,t] : \R^d)$ and $\gamma_n(t)=z$ such that $f_t(\gamma_n)\rightarrow v(t,z)$. Since $F$ and $v_0$ are bounded from above  and $v(t,z)$ is well defined, there is a constant $C=C(t,z)$ such that
$$
 \int_0^t|\dot\gamma_n|^2\leq C.
$$
Notice also that for any $s,s'\in[0,t]$ it holds that
\begin{equation}\label{eq:bound_size_trajectory}
 |\gamma_n(s)-\gamma_n(s')|\leq \int_0^t|\dot\gamma_n(r)|dr\leq \sqrt{t}\Big( \int_0^t|\dot\gamma_n(r)|^2dr\Big)^{1/2}\leq C\sqrt{t},
\end{equation}
therefore $(\gamma_n)$ is a 1/2-Holder continuous family of functions. Consequently, using that $\gamma_n(t)=z$, we deduce that
$$
\|\gamma_n\|_{W^{1,2}[0,t]}\leq C,
$$
taking a larger $C$ if necessary. In particular, there is some $\overline\gamma\in W^{1,2}([0,t] : \R^d)$ such that $\gamma_n\xrightarrow{n\rightarrow\infty}\overline\gamma$ strongly in $C([0,t] : \R^d)$ and weakly in $W^{1,2}([0,t] : \R^d)$. In particular
$$
v_0(\gamma_n(0))\rightarrow v_0(\overline\gamma(0)),\quad \int_0^tF(s,\gamma_n(s))ds\rightarrow\int_0^t F(s,\overline\gamma(s))ds
$$
and
$$
\int_0^t|\dot{\overline\gamma}(s)|^2ds\leq\liminf_{n\rightarrow\infty}\int_0^t|\dot\gamma_n(s)|^2ds.
$$
Thus,
\begin{equation}
\label{eq:overgamma}
 v(t,x)=v(\overline\gamma(0))+\int_0^t\Big(-\frac{|\dot{\overline\gamma}(s)|^2}4+F(s,\overline\gamma(s))\Big)ds.
\end{equation}

\textbf{Step 2. Uniqueness of the optimal trajectory in the set $\overline \Omega_0$; optimal trajectories remain in the set $\overline \Omega_0$.}

Let $z\in \overline \Omega_0$. Let $\overline \gamma$ be an optimal trajectory such that $\overline \gamma(t)=z$. We prove that such optimal trajectory has always been inside the set $\overline \Omega_0$. In other words, trajectories can leave but not enter, the set $\Omega_0$. We prove this by contradiction. Let's suppose that  {there exists} $s_1\in [0,t)$ such that $\overline \gamma (s_1)\in \R^d\setminus \overline \Omega_0$, and define 
$$
{t_1=\inf \{s\in [s_1,t] \,:\, \overline \gamma(s)\in \overline \Omega_0 \}.}
$$
Note that by our assumption the infimum above is taken over a non-empty set. Define also
$$
\mathcal A= \{s\in [0,s_1] \,:\, \overline  \gamma(s)\in  \overline \Omega_0 \}.
$$
If this set is non-empty we define
$$
t_0=\sup \{s\in\mathcal A \}.
$$
By definition, $t_0<s_1<t_1$ and for all $s\in (t_0,t_1)$, $ \overline \gamma(s)\in \R^d \setminus \overline \Omega_0$.
 
If the set $\mathcal A$ is empty we define  a new trajectory $\widetilde \gamma \in W^{1,2}([0,T] : \R^d)$ as below
$$
\begin{cases}
\widetilde \gamma(s)=\bar \gamma(t_1)& \text{for $s\in [0,t_1]$},\\
\widetilde \gamma(s)=\bar \gamma(s)& \text{for $s\in (t_1,t]$}.
\end{cases}
$$

If $\mathcal A$ is non-empty we define
$$
\begin{cases}
\widetilde \gamma(s)=\overline \gamma(s)& \text{for $s\in [0,t_0]$},\\
\widetilde \gamma(s)=(\frac{t_1-s}{t_1-t_0})\overline \gamma(t_0)+(\frac{s-t_0}{t_1-t_0})\overline \gamma(t_1)& \text{for $s\in (t_0,t_1]$}, \\
\widetilde \gamma(s)=\overline \gamma(s)& \text{for $s\in (t_1,t]$}.
\end{cases}
$$

Note that from the convexity of $\Omega_0$ we deduce the convexity of the set $[0,t]\times \Omega_0$, and thus $\widetilde \gamma(s)\subset \overline \Omega_0$ for all $s\in[t_0,t_1]$. Note also that in both of the above cases, using $v_0(\widetilde\gamma(0))\geq v_0(\bar\gamma(0))$ from assumption \eqref{eq:hyp3},  since the straight lines are local maximizers of the functional
\[
\int_0^t-\frac{|\dot{\overline\gamma}(s)|^2}4ds,
\]
and thanks   to assumption \eqref{hyp2}, we obtain $f_t(\overline \gamma)<f_t(\widetilde \gamma)$ which is in contradiction with the optimality of $\overline \gamma$. We hence obtain that $\overline \gamma(s) \subset \overline \Omega_0$ for all $s\in [0,t]$.

We next prove that such optimal trajectory $\overline\gamma$ is unique. To this end, we note
 that such trajectory $\overline\gamma$ satisfies the following Euler-Lagrange equation
\begin{equation}
\label{eq:EL}
\begin{cases}
 \ddot{\overline\gamma}(s) = -2D F(s,\overline\gamma(s)),\\
 \dot{\overline\gamma}(0)=-2D v_0(\overline\gamma(0)),\\
  \overline\gamma(t)=z.
 \end{cases}
\end{equation}
Moreover, $\gamma (s) \subset \overline \Omega_0$ for all $s\in [0,t]$ and hence $F$ and $v_0$ are strictly concave functions at the above points. Therefore, the above elliptic equation is coercive and the solution $\overline \gamma(s)$ is unique.
\\
 
\textbf{Step 3. Local regularity.}

 Take any $(t,z)\in  [0,T]\times\Omega_0$ and define $\gamma_z(\cdot):[0,T]\to \R^d$ the unique solution of~\eqref{eq:EL} which is itself, an optimal trajectory for the maximization problem. Note that thanks to Step 2, $\gamma_z(s)\in \overline{\Omega}_0$, for all $s=[0,T]$.  Using the regularity of $F$ and $v_0$, it follows that the functions
$$
(t,z)\rightarrow \gamma_z(t),\qquad (t,z)\rightarrow \dot\gamma_z(t)
$$
belong to $L^\infty_{loc}\big((\R^+\times\R^d) : W^{2,\infty}_{loc}(  \Omega_0)\big)$. Now, from~\eqref{eq:overgamma}, we have
$$
 v(t,x)=v_0(\gamma_x(0))+\int_0^t\Big(-\frac{|\dot\gamma_x(s)|^2}4+F(s,\gamma_x(s))\Big)ds,
$$
yielding
$$
 \partial_i v(t,x) = D v_0(\gamma_z(0))\cdot\partial_i\gamma_z(0)+\int_0^t\Big(-\frac{\dot\gamma_z(s)\cdot\partial_i\dot\gamma_z(s)}2+D F(s,\gamma_x(s))\partial_i\gamma_z(s)\Big)ds,
$$
integrating by parts in the first term of the integral, using $\int_0^t u {\rm d}v=[uv]_0^t - \int_0^t v{\rm d}u$ with $u(s)=-\dot{\gamma}_z(s)$ and $v=\partial_i \gamma_z(s)$, we obtain
$$
\begin{aligned}
\partial_i v(t,x) =& D v_0(\gamma_z(0))\cdot\partial_i\gamma_z(0) - \frac{\dot{\gamma}_z(t)\partial_i \gamma_z(t)}{2} + \frac{\dot{\gamma}_z(0)\partial_i \gamma_z(0)}{2}\\
&+\int_0^t\Big(\frac{\ddot\gamma_z(s)\cdot\partial_i\gamma_z(s)}2+D F(s,\gamma_x(s))\partial_i\gamma_z(s)\Big)ds.
\end{aligned}
$$
Using now~\eqref{eq:EL} and
\[
\partial_i \gamma_z(t)= (0,..., \underbrace{1}_{i-th},...,0),\quad\text{since}\quad \gamma_z(t)=z,
\]
it follows that
\[
D v(t,z) = -\frac{\dot\gamma_z(t)}2.
\]
In essence, the function $v$ restricted to $\Omega_0$ belongs to $L^\infty_{loc}(\R^+ : W^{3,\infty}( \Omega_0)\big)$. Finally, thanks to the compact embedding of $W^{2,\infty}(\Omega_0)$ in $C^1(\Omega_0)$ and the continuous differentiability of $\dot\gamma_z$ with respect to $t$ and $z$, we have that $D v \in C^1([0,T]\times \Omega_0)$.\\

\textbf{Step 4. Strong concavity.}

 We prove that $v$ is strictly concave in the set $\R^+\times \overline \Omega_0$. To this end, we show that for all $\sigma\in[0,1]$ and all $x,y\in \overline\Omega_0$ it holds that
$$
\sigma v(t,x)+(1-\sigma)v(t,y)+\lambda\sigma(1-\sigma)|x-y|^2\leq v\big(t,\sigma x+(1-\sigma)y\big),
$$
with $\lambda$ to be fixed later. This will lead to an upper bound for the second derivative of $v$ that is strictly negative and which depends only on $D^2v_0$ and $F$.

Let $\gamma_x$ and $\gamma_y$ be the optimal trajectories, solving~\eqref{eq:EL} with $\gamma_x(t)=x$ and $\gamma_y(t)=y$.   Using~\eqref{eq:overgamma} we have
\begin{equation}
\nonumber
 v(t,x)=v_0(\gamma_x(0))+\int_0^t\Big(-\frac{|\dot{\gamma}_x(s)|^2}4+F(s,\gamma_x(s))\Big)ds,
\end{equation}
and
\begin{equation}
\nonumber
 v(t,y)=v(\gamma_y(0))+\int_0^t\Big(-\frac{|\dot{\gamma}_y(s)|^2}4+F(s,\gamma_y(s))\Big)ds.
\end{equation}
Take $\sigma\in[0,1]$. Using~\eqref{eq:overgamma} at the point $(t,\sigma\gamma_x(t)+(1-\sigma)\gamma_y(t))=(t,\sigma x+(1-\sigma)y)$ we obtain that  
\begin{multline*}
 v(t,\sigma x+(1-\sigma)y)\geq v_0\big(\sigma\gamma_x(0)+(1-\sigma)\gamma_y(0)\big)
 \\
 +\int_0^t\Big(-\frac{|\sigma\dot\gamma_x(s)+(1-\sigma)\dot\gamma_y(s)|^2}4+F(s,\sigma\gamma_x(s)+(1-\sigma)\gamma_y(s))\Big)ds.
\end{multline*}
Now we use the fact that, thanks to Step 2,   for all $s\in [0,t]$, $\gamma_x(s)$ and $\gamma_y(s)$ lie inside $\overline\Omega_0$. Moreover, since both $v_0$ and $F$ are strictly concave on that subset, thanks to~\eqref{hyp1} and~\eqref{as:u0concave}, it follows that, for a certain positive constant $\overline L_1$,
$$
\sigma v_0(\gamma_x(0))+(1-\sigma)v_0(\gamma_y(0))+\overline L_1\sigma(1-\sigma)|\gamma_x(0)-\gamma_y(0)|^2\leq v_0(\sigma \gamma_x(0)+(1-\sigma)\gamma_y(0)).
$$
and on the other hand, for a certain positive constant $\overline K_1$,
\begin{multline*}
 \sigma\int_0^tF(s,\gamma_x(s))ds+(1-\sigma)\int_0^tF(s,\gamma_y(s))ds+ \overline K_1\sigma(1-\sigma)\int_0^t|\gamma_x(s)-\gamma_y(s)|^2ds\\
 \leq \int_0^tF(s,\sigma\gamma_x(s)+(1-\sigma)\gamma_y(s))ds.
\end{multline*}
And combining both inequalities, we get that
\begin{multline*}
 v(t,\sigma x+(1-\sigma)y)\geq \sigma v(t,x)+(1-\sigma)v(t,y)+\overline L_1\sigma(1-\sigma)|\gamma_x(0)-\gamma_y(0)|^2
 \\
 +
 \int_0^t\Big(-\frac{|\sigma\dot\gamma_x(s)+(1-\sigma)\dot\gamma_y(s)|^2}4+\sigma(1-\sigma)\overline K_1|\gamma_x(s)-\gamma_y(s)|^2\Big)ds
 \\
 +\sigma\int_0^t\frac{|\dot{\gamma}_x(s)|^2}4ds+(1-\sigma)\int_0^t\frac{|\dot{\gamma}_y(s)|^2}4ds.
\end{multline*}

 Finally, recalling that the map $-|\cdot|^2$ is strictly concave, we deduce that
\begin{multline*}
 v(t,\sigma x+(1-\sigma)y)\geq \sigma v(t,x)+(1-\sigma)v(t,y)
 \\
 +
 \sigma(1-\sigma)\Big[\int_0^t\Big(\frac14|\dot\gamma_x(s)-\dot\gamma_y(s)|^2+\overline K_1|\gamma_x(s)-\gamma_y(s)|^2\Big)ds+\overline L_1|\gamma_x(0)-\gamma_y(0)|^2\Big].
\end{multline*}
To conclude, we use Young's inequality (also known as Peter--Paul inequality) to notice that
\begin{multline*}
 |x-y|^2   =  |\gamma_x(0)-\gamma_y(0)|^2+\int_0^t\frac{d}{ds}|\gamma_x(s)-\gamma_y(s)|^2ds
 \\
 \leq|\gamma_x(0)-\gamma_y(0)|^2+{4\lambda}\int_0^t|\gamma_x(s)-\gamma_y(s)|^2ds+\frac{1}{4\lambda}\int_0^t|\dot\gamma_x(s)-\dot\gamma_y(s)|^2ds,
\end{multline*}
for all positive $\lambda$. We then choose  $\lambda=\min\big(\overline L_1,\sqrt{\overline K_1}/2\big)$ to obtain that
$$
\sigma v(t,x)+(1-\sigma)v(t,y)+\lambda\sigma(1-\sigma)|x-y|^2\leq v\big(t,\sigma x+(1-\sigma)y\big).
$$
We conclude, thanks to the regularity of $v$, that
\[
x\ D^2 v(t,z)\ x^T\leq -2\lambda |x|^2\quad\text{for all } (t,z)\in [0,T]\times\Omega_0,\ x\in \R^d.
\]

\textbf{Step 5. Semi-convexity.}

Let $E\subset \R^d$ be a bounded domain. Then there exists a positive constant $S_c(E, \underline{K}_1, \underline{L}_1)$ such that for all $t\in [0,T]$ and for any $x\in\R^d$,
$$
x \, D^2v(t,z)\,  x^T\geq -S_c|x|^2\quad\text{for all}\quad z \in E,
$$
where $\underline{K}_1, \underline{L}_1$ come from hypothesis~\eqref{hyp1} and~\eqref{as:u0concave}.

We will employ the Bernstein Method, and in order to do so we start with a viscous version of the problem~\eqref{eq:1}, considering a slightly modified equation
\begin{equation}\label{eq:viscous_v_convexity}
\partial_t v = \varepsilon\Delta v  + 
 |D v|^2+F(t,z),\quad   t\in [0, T],z\in\R^d.
\end{equation}
Notice that this solution $v$ will depend on $\varepsilon$ but the semi-convexity will be uniform.\\

Let us consider a closed set $E'$ such that $E\subset E'\subset \R^d$ and $\mathrm{dist}(E,\p E')>1$, differentiate twice equation~\eqref{eq:viscous_v_convexity} with respect to the variables $z_i$ and $z_j$, and then multiply it by the square of a cut-off function $\varphi$ satisfying
$$
\varphi(t,z)=\frac{t+1}{T+1}\phi(z),
$$
where $\phi$ is a positive function such that, for all $\varepsilon\in (0,1)$,
\[
\begin{aligned}
	&\text{(i) }0\leq \phi\leq 1, \phi\equiv 1\text{ for all }z\in E, \phi\equiv 0\text{ for all }z\not\in E',\\
	&\text{(ii) } \|Dv \cdot D\phi\|_{L^\infty(E')}< 1/4,\\
	&\text{(iii) }  \phi'' - 2  \frac{(\phi')^2}{\phi}\in [-1,1].
\end{aligned}		
\]

If we define $v^i:=\partial_{z_i} v$ and $v^{i,j}:=\partial^2_{z_i z_j}v$ we obtain
$$
v^{i,j}_t\varphi^2 =\varepsilon \Delta v^{i,j}\varphi^2+2Dv^i \cdot Dv^j\varphi^2+2Dv \cdot Dv^{i,j}\varphi^2+F^{i,j}\varphi^2,
$$
which we can write down as
\[
\partial_t (D^2v)\varphi^2 =\varepsilon \mathcal{L}(v)\varphi^2+2\mathcal{D}(v)\varphi^2+2\mathcal{M}(v)\varphi^2+D^2F\varphi^2,
\]
where we defined the matrices $\mathcal{L}(v), \mathcal{D}(v)$ and $\mathcal{M}(v)$ taking their entries as
\[
\mathcal{L}(v)_{i,j}:= \Delta v^{i,j},\quad \mathcal{D}(v)_{i,j}:= Dv^{i} \cdot Dv^j\quad\text{and}\quad \mathcal{M}(v)_{i,j}:= Dv\cdot Dv^{i,j}.
\]
We choose now an arbitrary vector $x\in\R^d$ and multiply the equation by $x$ by the left and by $x^T$ (the transposed vector) by the right, obtaining
\begin{equation}\label{eq:bernstein_with_matrices}
	\begin{aligned}
	\left[x\, \partial_t (D^2v)\, x^T\right]\varphi^2 =&\varepsilon \left[x\,\mathcal{L}(v)\,x^T\right]\varphi^2+2\left[x\,\mathcal{D}(v)\, x^T\right]\varphi^2+2\left[x\,\mathcal{M}(v)\, x^T\right]\varphi^2\\[10pt]
	&+\left[x\, D^2F\, x^T\right]\varphi^2.
	\end{aligned}
\end{equation}

After this, we define
\[
w(t,z):=  \left[x\, D^2v(t,z)\, x^T\right] \varphi(t,z).
\]

If $w$ attains its minimum at $t=0$, then, by assumption~\eqref{as:u0concave}
\[
w(t,z)\geq w(0,z)= \left[x\, D^2v_0(z) x^T\right]\varphi(0,z)\geq -2\underline{L}_1|x|^2\phi(z)\geq -2\underline{L}_1|x|^2
\]
and the claim is proven. If not, then at a point of minimum of $w$ (which must be attained inside the support of $\varphi$) we must have $w_t\leq 0,Dw=0$ and $\Delta w\geq 0$. Rewriting   this   in terms of $v$ and $\varphi$ we obtain
\[
\begin{aligned}
\left[x \,\partial_t (D^2v)  \,x^T\right]\varphi&\leq -\left[x\,(D^2v)\, x^T\right]\partial_t\varphi,\\[10pt]
 D\left(\left[x\,  (D^2v)\,  x^T\right]\right)\varphi &= - \left[x\, (D^2v)\, x^T\right]D\varphi, \\[10pt]
 \left[x \, \mathcal{L}(v) \, x^T\right]\varphi &\geq \left[x\, (D^2v)\, x^T\right]\left(2\frac{|D\varphi|^2}{\varphi} - \Delta\varphi\right).
\end{aligned}
\]
We next use the second equation above to obtain that 
\[
 \left[x\, \mathcal{M}(v)\, x^T\right]\varphi  = -\left[x\, (D^2v)\, x^T\right]Dv \cdot D\varphi.
\]
We also compute, using the symmetry of the matrices $(x^T \, x)$ and $D^2v$, 
\[
\begin{aligned}
\left[x \, (D^2v) \, x^T\right]^2 &= \left[x \, (D^2v) \, (x^T \, x) \, (D^2v) \, x^T\right]\\[10pt]
&= \left[x \, (D^2v) \, (D^2v) \, (x^T\, x) \, x^T\right]\\[10pt]
&= \left[x \, (D^2v)^2 \,  x^T\, (x \cdot x^T)\right] .
\end{aligned}
\]
Next we use the identities
\[
(D^2v)^2 = \mathcal{D}(v)\quad\text{and}\quad x\cdot x^T=|x|^2 ,
\]
to obtain that 
\[
\left[x \, (D^2v) \, x^T\right]^2=|x|^2\left[x \, (D^2v)^2 \, x^T\right] = |x|^2\left[x\ \mathcal{D}(v)\ x^T\right].
\]
We deduce that
\[
 \left[x \, \mathcal{D}(v) \, x^T\right] =\left[x (D^2v)\, x^T\right]^2|x|^{-2} .
\]
Substituting all this into equation~\eqref{eq:bernstein_with_matrices} we obtain
$$
-w\varphi_t \geq - \varepsilon w\Delta\varphi+2 \varepsilon w\frac{|D\varphi|^2}{\varphi}+2w^2|x|^{-2}-2wDv \cdot D\varphi+\left[x\cdot D^2F\cdot x^T\right]\varphi^2.
$$
We can rearrange here, recalling that by hypothesis~\eqref{hyp1} the last term is bounded below by a negative constant $-2\underline{K}_1$, to obtain
$$
2w^2|x|^{-2}\leq \left(-\varphi_t +\varepsilon \Delta\varphi - 2 \varepsilon \frac{|D\varphi|^2}{\varphi}+ 2Dv\cdot D\varphi\right)w + 2\underline{K}_1|x|^2.
$$

With the properties of $\varphi$ in mind one can see that, at least for $\varepsilon$ small enough,
$$
2w^2\leq f(t)|x|^2 w + 2\underline{K}_1|x|^4\quad\text{with}\quad f(t)\geq -1-\frac{1}{T+1},
$$
implying that $w\geq |x|^2(f-\sqrt{f^2+16\underline{K}_1})/4\geq -C(E, \underline{K}_1)|x|^2$ for a certain positive constant $C$. Therefore
\[
w(t,z)\geq \min\{ -2\underline{L}_1,\ -C\}|x|^2
\]
We define $S_c(E, \underline{K}_1, \underline{L}_1)$ to be the minimum above. The result follows for a fixed $\varepsilon$ since $\varphi(T, z)=1$ for all $z\in E$, and then we let $\varepsilon\to 0$.

\textbf{Step 6. $v(t,\cdot)$ has a local maximum in $\Omega_0$.}

From Step 4 we know that $v(t,\cdot)$ is strictly concave in $\Omega_0$. Consequently, for all $t\in [0,T]$, $v(t,\cdot)$ has a local maximum point $\z(t)$ in  $\overline \Omega_0$. Moreover, the continuity of $v$ implies that $\bar z$ is also continuous. We prove by contradiction that for all $t\in [0,T]$, we have indeed $\z(t) \in   \Omega_0$.  Define $t_1$ as the first time  in $[0,T]$ such that $\bar z(t)\in \partial \Omega_0$. Thanks to~\eqref{eq:hypz0} and the continuity of $\z$ we obtain that $t_1>0$. We show that   $t_1$ cannot exist.

We use the regularity of $v$ in $\Omega_0$ obtained in Step 3 to differentiate the equation in~\eqref{eq:1} with respect to $z$ and evaluate it at $\z(t)$: 
$$
\partial_t( {D} v)(t,\z(t))=2 D^2 v(t,\z(t))  \,{D}  v(t,\z(t))  + {D} F(t,\z(t)).
$$
We also note that since $\z(t)$ is a maximum point of $v$, we have
\begin{equation}
\label{nablau0-1}
{D} v(t,\z(t)) =0,
\end{equation}
and hence 
$$
\partial_t( {D} v)(t,\z(t))= {D} F(t,\z(t)).
$$
We next differentiate~\eqref{nablau0-1} with respect to $t$ to obtain
$$
\partial_t({D} v)(t,\z(t))+D^2 v(t,\z(t))\, \dot{\z}(t) =0.
$$
We then combine the  equalities above and the strict concavity of $v$ in $\Omega_0$ to obtain 
\begin{equation}
\label{can-eq}
 \dot{\z}(t) = \big(-D^2 v(t,\z(t))\big)^{-1}\, {D} F\big(t,\z(t)\big), \qquad \forall t\in [0,t_1).
\end{equation}
Now let's assume that $t_1\leq T$, with $\z(t_1) \in \p \Omega_0$. We multiply~\eqref{can-eq} by ${D} F\big(t,\z(t)\big)$, for $t\in [0,t_1)$, to obtain, thanks to the strict concavity of $v$ in $\overline \Omega$,   that there exists $\lambda>0$ such that 
$$
{D} F\big(t,\z(t)\big) \cdot\dot{\z}(t) \geq \lambda |{D} F\big(t,\z(t)\big)|^2,  \qquad \forall t\in [0,t_1).
$$
We then notice  that since $\z \in C([0, T] : \overline \Omega_0)$,   ${D} F\big(t,\z(t)\big)$ is a continuous function. Therefore, thanks to~\eqref{as:nablaR0} and the continuity of $\dot{\z}$, there exists $t_2$ and a $\nu>0$ such that $0<t_2<t_1$ and
$$
 {D} F\big(t_2,\z(t)\big)\cdot \dot{\z}(t) \geq \nu,  \qquad \forall t\in [t_2,t_1).
$$
We then integrate the above inequality with respect to $t$ in the interval $(t_2,t_1)$ to obtain
$$
F\big( t_2 ,\z(t_1)\big)- F\big( t_2 ,\z(t_2)\big)  \geq \nu(t_1-t_2)>0. 
$$
This is in contradiction with assumption~\eqref{hyp2}  and the facts that $\z(t_1)\in \p \Omega_0$ and $\z(t_2)\in \Omega_0$.

We highlight that, up to this point, we have not made use of the constraint $\max\limits_{z\in\R} v =0$ from~\eqref{eq:1}.

\textbf{Step 7. The unique  local maximum $\z(t)$ of $v(t,\cdot)$ in $\Omega_0$ is indeed a global maximum.}

 We have already proven that, for all $t\in [0, T]$, $v(t,\cdot)$ has a unique local maximum point in $\overline \Omega_0$ attained at $\z(t)\in C^1([0,T] : \Omega_0)$, though we might run into the case where $v(t,\z(t))<0$ for some $t\in [0,T]$. In this case, in view of the second equality in \eqref{eq:1}, there must exist  $t_3\in [0,T]$ and a point $\bar{w}(t_3)\in \R^d\setminus\Omega_0$   such that
 \[
 v(t, \bar{z}(t_3))<0  \quad\text{and}\quad v(t, \bar{w}(t_3))=0.
 \]  
 Then, since $v$ is in particular a viscosity subsolution of~\eqref{eq:1} and
 \[
 \max_{(t,z)\in[0,T]\times \R} v(t, z)=v(t_3, \bar w(t_3))=0,
 \]
 we can choose a constant test function $\varphi\equiv 0$ to apply the subsolution criterion at $(t_3, \bar w(t_3))$.  We obtain indeed that 
\[
{0=}\partial_t\varphi(t_3, \bar w(t_3)) - |\partial_z \varphi(t_3, \bar w(t_3))|^2\leq F(t_3, \bar w(t_3)).
\]
{Consequently} and thanks to~\eqref{hyp2}, we obtain 
\begin{equation}
\label{ineqt3}
0\leq F(t_3,\bar w(t_3)) < F(t_3,\z(t_3)).
\end{equation}
Furthermore, using the regularity of $v$ in $[0,T]\times \Omega_0$ we can  evaluate the first equation in \eqref{eq:1} at $\z(t_3)$ to obtain that 
$$
F(t_3,\z(t_3))=0
$$
which is in contradiction with \eqref{ineqt3}.

Note however that we still have the possibility of having a second simultaneous maximum point outside $\Omega_0$, but we will prove it false on the next step.

\textbf{Step 8. The maximum of $v(t,\cdot)$ is attained only at the point $\overline z(t)\in \Omega_0$.}

By the strict concavity of $v(t,\cdot)$ in $\Omega_0$ and the fact that $\z (t)\in \Omega_0$ we deduce that there is no other maximum point in $\overline \Omega_0$. We prove that this also holds in $\R^d\setminus\overline \Omega_0$. Let us argue by contradiction. Let $t_1\in (0, T]$ be the smallest time such that there exists a point $w\in \R^d\setminus\overline \Omega_0$ such that $v(t_1, w)=0$, and let $\gamma_{w}$ be any continuous curve (which may not be unique) such that
\begin{equation}
\label{DP-v}
v(t_1, w)=v_0(\gamma_{w}(0))+\int_0^{t_1}\Big(-\frac{|\dot\gamma_{w}(s)|^2}4+F(s,\gamma_{w}(s))\Big)ds.
\end{equation}
Since $v$ and $\gamma_{w}$ are continuous, it means that there must exist a time $t_0$ such that for all $t\in [t_0, t_1]$, $\gamma_{w}(t)\in \R^d\setminus\overline\Omega_0$ and $v(t_0,\gamma_{w}(t_0))<0$. Moreover, since $F(s,\bar z(s))=0$ for all $s\in [t_0, t_1]$,  and in view of Assumption \eqref{hyp2}, we obtain that for all $s\in [t_0, t_1]$,
\[
F(s,\gamma_{w}(s))<0.
\]
Then, using \eqref{DP-v}  
we obtain that $v(t_1, w)<0$, a contradiction.

{\bf Step 9. Bounds on $D^3 v$.}

We prove that $D^3 v$ is bounded uniformly in $x\in  \Omega_0$ and locally in $t$. 
The proof of this bound follows from similar arguments to  \cite{Mirrahimi-Roquejoffre - 2016}. We set $w(t,x)=D^3 v(t,x)$, which solves
\[
\p_t w-2D w\cdot D v=S(t,x,w)=6w\cdot D^2v +D^3F(s,x),
\]
where $D w$ denotes the column of tensors $(\p_1 w,\dots,\p_d w)$ and $w\cdot D^2 v$ denotes the column of matrices $(\p_1 D^2 v\cdot D^2v,\dots, \p_d D^2 v\cdot D^2 v)$. This is a linear transport equation with bounded coefficients in $[0,T]\times \overline \Omega_0$ (thanks to the bound on $D^2 v$). Moreover, the characteristics corresponding to this equation with ending point in $[0,T]\times \overline{\Omega}_0$ have always remained in $\overline{\Omega}_0$ thanks to step 2 (i.e. $\gamma(s)\in \Omega_0$, for all $s\in [0,T]$). Hence the desired bound on $\|D^3 v\|_{L^\infty_{\mathrm{loc}}(\R^+\times\Omega_0)}$.

\subsection{Proof of Theorem \ref{thm:section_Dyn_Prog_2}}\label{sec:Dyn_Prog_2}

Theorem \ref{thm:section_Dyn_Prog_2} can be proven thanks to the properties obtained in the previous section and following similar arguments as in~\cite{Mirrahimi-Roquejoffre - 2016}. We only comment the proof of the uniqueness of the solution to \eqref{eq:1}, with $F(t,z)=G(z,\z(t))$, which requires an additional small argument. The main ingredient for the uniqueness of the viscosity solution to \eqref{eq:1} is indeed the equivalence of \eqref{eq:1} with  \eqref{eq:3}. Then, the uniqueness follows using a fixed point argument which relies on the following technical lemma.	
\begin{lem}
Let $v_i$, for $i=1,2$, be the viscosity solution to
\[
\begin{cases}
\p_t v_i(t,z)=|Dv_i(t,z)|^2+G(z,\z_i(t)),& (t,z)\in (0,T)\times \R^d \\
\dot\z_i(t)=(-D^2v_i(t,\z_i(t)))^{-1}\ DG_i(\z_i(t), \z_i(t)),& t\in (0,T),\\
v_i(0,z)=v_{0,i}(z), & z\in \R^d,\\
\z_i(0)=z_{0,i}\in\Omega_0.
\end{cases}
\]
Then, we have
\[
\|v_1-v_2\|_{L^\infty((0,t);W^{3,+\infty}(\Omega_0)}\leq C\|\z_1-\z_2\|_{L^\infty([0,t])}+C|\|v_{0,1}-v_{0,2}\|_{ W^{3,+\infty}(\Omega_0)}.
\]
\end{lem}
\begin{proof}
The proof of this lemma follows from similar arguments as in Lemma 4.2 in \cite{Mirrahimi-Roquejoffre - 2016}. We only provide an additional  argument which is needed in our case. Let $r=v_1-v_2$. Then, $r$ satisfies
\[
\p_t r =(Dv_1+Dv _2)\cdot Dr +G(z,\z_1(t))-G(z,\z_2(t)).
\]
This is a linear transport equation. 
Notice also that thanks to Assumption \eqref{asRD23_G}, we have 
\[
\|G(z,\z_1(t)-G(z,\z_2(t))\|_{W^{2,\infty}_z(\overline \Omega_0)}\leq K_3|\z_1(t)-\z_2(t)|.
\]
In order to obtain the result, it is hence enough  (see the proof of  Lemma 4.2 in \cite{Mirrahimi-Roquejoffre - 2016} for more details) to show that the characteristics, defined by 
\[
\begin{cases}
\dot{\gamma}(t)=-Dv_1(t,\gamma)- Dv_2(t,\gamma),\\
\gamma(t)=x,\quad \text{with $x\in \overline{\Omega}_0$},
\end{cases}
\]
has never left the set $\overline{\Omega}_0$, that is $\gamma(s)\in \overline{\Omega}_0$, for all $s\in [0,t]$. This property holds true thanks to the strict concavity of $v_1$ and $v_2$ in $[0,t]\times \overline{\Omega}_0$ and since the maximum of $v_i$ is attained in $\overline \Omega_0$ thanks to Theorem \ref{thm:section_Dyn_Prog_1}.
\end{proof}

\section{Proof of Theorem~\ref{thm:limit_epsilon}}\label{section:limit_varepsilon}
	
	Replacing the Hopf-Cole transformed function~\eqref{eq:transformed_hopf_cole} into equation~\eqref{eq:main}	we obtain
	\begin{equation}\label{eq:hopf-cole}
		\begin{cases}
			\partial_t u_\varepsilon(t,z)=\varepsilon \partial^2_{zz} u_\varepsilon(t,z) +\left| \partial_{z}u_\varepsilon(t,z) \right|^2 + R(z) - \rho_\varepsilon(t) +{\Phi_\varepsilon(t,z)},\\
			\rho_\varepsilon(t) = \displaystyle\int_\R n_\varepsilon(t,z){\rm dz},
		\end{cases}
	\end{equation}
	{where
	$$
	\Phi_\varepsilon(t,z):=\tau \displaystyle\int_\R \frac{n_\varepsilon(t,y)}{\rho_\varepsilon(t)}H(z-y)\ {\rm d}y.
	$$}
	Note that thanks to Assumption \eqref{eq:hypothesis_H}, $\Phi_\e(t,z)\in [-\tau,\tau]$, $\partial_z \Phi_\varepsilon(t,z)\in [0,\tau]$ and $\partial_{zz}^2\Phi_\varepsilon(t,z)\in [-\tau\cdot\sup H'', \tau\cdot\sup H'']$.

\subsection{Regularity estimates}

Let us begin by presenting an initial convergence result for $\rho_\varepsilon$ and $\phi_\varepsilon$.
\begin{lem}\label{lem:rho_phi_weak_convergence}
For all $\varepsilon\in (0,1)$ we have that
\begin{equation}
\label{bound-rho-e}
0\leq \rho_\varepsilon(t)\leq \rho_{\text{max}}\quad \text{and}\quad-\tau<\Phi_\varepsilon(t,z)<\tau,
\end{equation}
with
\[
\rho_{\text{max}}:=\max\{\max\limits_{z\in\R}R(z), \rho_M\}=\max\{1, \rho_M\}.
\]
As a consequence, there exist  functions $\rho(t)\in L^\infty(\R^+)$ and $\Phi(t,z)\in  L^\infty(\R^+;C^2( \R))$ such that, as $\varepsilon\to 0$ and along subsequences,
	\[
	\rho_\varepsilon(t)\to  \rho(t) \in [0,\rho_{\text{max}}]	,
	\]
	weakly-* in $L^\infty(\R_+)$ and $\Phi_\varepsilon(t,z)$ converges to $\Phi (t,z) $ in $L^\infty(w\ast (\R^+);C^2( \R))$ with 
	\[
	\begin{aligned}
			   \Phi(t,z) 	\in [-\tau, \tau],\quad      \partial_z\Phi(t,z) 	\in [0, \tau] \quad\text{and} \quad \partial^2_{zz}\Phi(t,z)	\in [-\tau\cdot\sup H'', \tau\cdot\sup H''].
	\end{aligned}
	\]
\end{lem}
\begin{proof}
Integrating the equation~\eqref{eq:main} in the whole $\R$  (this integration should be done via a test function, but we omit the details) we obtain
	\[
	\varepsilon \rho_\varepsilon'(t)=-\rho_\varepsilon(t)^2 +\int_\R n_\varepsilon(t,z)R(z){\rm dz} + \int_\R n_\varepsilon(t,z)\Phi_\varepsilon(t,z){\rm dz}.
	\]
	However, the last integral is equal to 0 due to the symmetry of the kernel $H$, let us see how. We have
	\[
	\begin{aligned}
		\int_\R n_\varepsilon(t,z)\Phi_\varepsilon(t,z){\rm dz} &=\frac{\tau}{\rho(t)} \int_\R\int_\R n_\varepsilon(t,z) n_\varepsilon(t,y)H(z-y){\rm dy}{\rm dz}\\
		&=- \frac{\tau}{\rho(t)} \int_\R\int_\R n_\varepsilon(t,z) n_\varepsilon(t,y)H(y-z){\rm dy}{\rm dz}\\
		&=- \frac{\tau}{\rho(t)} \int_\R\int_\R n_\varepsilon(t,z) n_\varepsilon(t,y)H(y-z){\rm dz}{\rm dy}\\
		&=- \frac{\tau}{\rho(t)} \int_\R\int_\R n_\varepsilon(t,y) n_\varepsilon(t,z)H(z-y){\rm dy}{\rm dz}= - \int_\R n_\varepsilon(t,z)\Phi_\varepsilon(t,z){\rm dz},
	\end{aligned}
	\]
	where we used first the symmetry of $H$, then applyed Fubini's Theorem, and finally renamed $z$ as $y$ and vicerversa. Since $x=-x$ if and only if $x=0$, we conclude our claim. Therefore
	\[
	\varepsilon \rho'_\varepsilon(t)\leq \max\limits_{z\in\R}R(z)\cdot \rho_\varepsilon(t)-\rho_\varepsilon(t)^2.
	\]
	This differential inequality yields the upper bound for $\rho$ in \eqref{bound-rho-e}. The rest of the lemma follows from the fact that  $\rho_\e$ and $\Phi_\e$ are bounded respectively in $L^\infty(\R^+)$ and $L^\infty(\R^+; C^3(\R))$.
\end{proof}

	For the study of equation~\eqref{eq:hopf-cole} we will need the following regularity estimates.
 
\begin{prp}\label{prp:regularity}
	Let conditions~\eqref{hyp:R},~\eqref{eq:hypothesis_H} and~\eqref{eq:hypothesis_u_0} be satisfied. Let $u_\varepsilon$ be a solution of~\eqref{eq:hopf-cole} and $\varepsilon\leq\varepsilon_0<1$ for a certain $\varepsilon_0\in(0,1)$.
	\begin{itemize}
		\item[(i)] There exist positive constants     $C_1$ and $C_2$ such that, for $\e$ small enough,
	\begin{equation}
	\label{uniform-bound-u}
	 -A_1-\overline B_1z^2-C_1t \leq u_\e(t,z)\leq  A_2-\overline B_2z^2+C_2t,
	\end{equation}
	with $\overline B_1=\max(B_1,\sqrt{K_4}/2)$ and $\overline B_2=\min(B_2,\sqrt{K_2}/2)$.  

\item[(ii)] Let $D\subset \R$ be a bounded domain and $T>0$. Then there exists a positive constant $C(R,\tau, D)$ such that for all $\varepsilon\in (0,1)$,
		$$
		\left| \partial_z u_\varepsilon(t,z) \right|\leq C(R,\tau, D)\quad\text{for all}\quad t\in [0, T],\ z\in D.
		$$
			
	\item[(iii)] Let $D\subset \R$ be a bounded domain. Then there exists a positive constant $S_c(R,\tau, D, C_2)$ such that for all $\varepsilon\in (0,1)$, $t\geq 0$,
	$$
	\partial^2_{zz}u_\varepsilon(t,z)\geq -S_c(R,\tau, D, C_2)\quad\text{for all}\quad z \in D,
	$$
	where $C_2$ comes from hypothesis~\eqref{eq:hypothesis_u_0}. In other words, the family $u_\varepsilon$ is locally uniformly semi-convex. 
    \item[(iv)] Let $T>0$. Then the functions $u_\varepsilon$ are uniformly locally equicontinuous with respect to time for all $t\in [0,T]$.
\end{itemize}
All of this together means that, if the initial datum $u_{\varepsilon,0}$ is regular enough, the family of functions $\{u_\varepsilon\}_\varepsilon$ is locally uniformly equibounded and equicontinuous with respect to $z$ and $t$, and locally uniformly semi-convex with respect to $z$.
\end{prp} 
\begin{proof}

	\noindent\textbf{Proof of Proposition~\ref{prp:regularity} (\textit{i})}:

 The result follows simply from a comparison principle applied to~\eqref{eq:main}. Notice that, thanks to assumptions \eqref{hyp:R}, \eqref{eq:hypothesis_H} and Lemma \ref{lem:rho_phi_weak_convergence}, we have
 $$
n_\e (K_3-K_4z^2-\rho_{\rm max}-\tau )\leq \varepsilon \partial_t n_\varepsilon(t,z)-\varepsilon^2 \partial^2_{zz}n_\varepsilon (t,z)\leq   n_\varepsilon(t,z)(K_1-K_2z^2+\tau).
 $$
Combining these inequalities with the hypothesis~\eqref{eq:hypothesis_u_0},  we can apply the comparison principle to deduce that, there exist positive constants     $C_1$ and $C_2$ such that, for $\e$ small enough,
$$
\exp\big(\f{-A_1-\overline B_1z^2-C_1t}{\e}\big)\leq n_\e\leq \exp\big(\f{A_2-\overline B_2z^2+C_2t}{\e}\big),
$$
with $\overline B_1=\max(B_1,\sqrt{K_4}/2)$ and $\overline B_2=\min(B_2,\sqrt{K_2}/2)$.

	\textbf{Proof of Proposition~\ref{prp:regularity} (\textit{ii})}:

	Let us apply the Bernstein Method and start by choosing a closed set $E$ such that $D\subset E\subset \R$ and $\mathrm{dist}(D,\p E)>1$, and a finite time $T$. We will study the equation in the set $E_T:= [0,T]\times E$. Let $m$ and $M$ be the minimum and maximum possible values of $u_\varepsilon$ on $E_T$ respectively. {These constants} can be chosen independently of $\e$ thanks to \eqref{uniform-bound-u}. We define {the auxiliary functions}
	$$
	\theta(x)=-\frac{M-m}{2\Lambda^2}x^2+\frac{3(M-m)}{2\Lambda}x+m\quad \text{for }x\in [0, \Lambda],\quad w(t,z):=\theta^{-1}(u_\varepsilon(t,z)),
	$$
	with $\Lambda>\max(3(M-m), 2)$, a big enough constant. {The idea here is to work with the function $w$ instead of  the function $u$. We will indeed derive an equation on $w$ that we will differentiate with respect to $z$. We then obtain the desired estimate after some computations on this equation and using the maximum principle. The function $\theta$ has the advantage of introducing a term of order $(w')^4$ with definite sign in equation~\eqref{eq:bernstein_theta_w},   that  would not appear in the computations otherwise (consider $\theta(w)=w$ to verify this). This term is crucial to obtain the desired estimate.} \\
	Since $u_\varepsilon=\theta(w)$ and $\theta$ is monotone increasing, it is clear that $w:E_T\to [0,\Lambda]$. Moreover, we have that
	$$
	|\theta'|\leq 1,\quad \frac{\theta'''\theta'-(\theta'')^2}{(\theta')^2}< -\theta''\quad\text{and}\quad \left| \frac{\theta''}{\theta'} + \theta'\right|< 1.
	$$
	A bound for $\partial_z w$ will yield immediately a similar bound for $\partial_z  u_\varepsilon$. For the rest of this section we will use the prime notation $w'$ instead of $\partial_z w$  for the sake of brevity.

	Finally, let us take a cut-off function $\zeta(t,z)$ such that
	$$
	\zeta(t,z):=\frac{t+1}{T+1}\varphi(z),
	$$
	where $\varphi$ is a positive smooth, compactly supported cut-off function such that
	$$
	\varphi(z)\equiv 0\text{ for all }z\in \R\setminus E,\quad \varphi(z)\equiv 1\text{ for all }z\in D
	$$
	and
	$$
	|\zeta|\leq 1,\quad |\zeta'|_\infty + |\zeta''|_\infty\leq C_\varphi,
	$$
	where $C_\varphi$ is a positive constant.
	
	We substitute $u_\varepsilon=\theta(w)$ in~\eqref{eq:hopf-cole} to obtain
	$$
	w_t=\left(\varepsilon \frac{\theta''}{\theta'} + \theta'\right)(w')^2 + \varepsilon w'' + h,
	$$
	where
	$$
	h(t,z):= \frac{1}{\theta'}\big(-\rho_\varepsilon(t) + R(z)+\Phi_\varepsilon(t,z)\big).
	$$
	Let us differentiate this equation with respect to $z$   and multiply by $w'\zeta^4$ in order to obtain
	$$
	w'w'_t\zeta^4=\left(\varepsilon \frac{\theta'''\theta'-(\theta'')^2}{(\theta')^2} + \theta''\right)(w')^4\zeta^4 +2\left(\varepsilon \frac{\theta''}{\theta'} + \theta'\right)(w')^2w''\zeta^4+ \varepsilon w' w'''\zeta^4  + h'w'\zeta^4.
	$$
	We next look  for a maximum of the function $v:=|w'\zeta|$. This maximum should be attained at an interior point of $E$ that we denote by $(t_m, z_m)$. If it is attained at $t_m=0$, then by the assumption~\eqref{eq:hypothesis_u_0} there must exist a constant 
	\[
	k_D:=\sup\limits_{z\in D,\  \varepsilon\in(0,1)}|\partial_zu_{0,\varepsilon}(z)|
	\]
	such that
	\begin{equation}\label{eq:k_tilde}
	|u'_\varepsilon(t,z)|\leq \tilde k_D\text{ for all }(z,t)\in [0,T]\times D,
	\end{equation}
	which is the desired result.
	
	If, on the contrary, the maximum point is attained at a positive time $t_m>0$, since the maximum of $v$ must coincide with the maximum of $v^2$, on such a point we must have
	$$
	[(w'\zeta)^2]'(t_m, z_m)=0\Rightarrow w''(t_m, z_m)\zeta(t_m, z_m)=-w'(t_m, z_m)\zeta'(t_m, z_m),
	$$
	and
	$$
	[(w'\zeta)^2]''\leq 0\Rightarrow w'''w'(z_m)\zeta^2\leq \left(2(\zeta')^2 - \zeta\zeta''\right)(w')^2,
	$$
	and also
	$$
	\partial_t[(w'\zeta)^2]\geq 0\Rightarrow w'w'_t\zeta^2\geq-(w')^2\zeta\zeta_t.
	$$
	We omitted the point $(t_m, z_m)$ for the sake of brevity. Replacing this in the previous equation we obtain that, at the maximum point $(t_m, z_m)$,
	\begin{equation}\label{eq:bernstein_theta_w}
	\begin{aligned}
		-(w')^2\zeta^3\zeta_t\leq & \left(\varepsilon\frac{\theta'''\theta'-(\theta'')^2}{(\theta')^2} + \theta''\right)(w')^4\zeta^4 -2\left(\varepsilon \frac{\theta''}{\theta'}+ \theta'\right)(w')^3\zeta^3\zeta' \\
		& + \left(\varepsilon (2(\zeta')^2-\zeta\zeta'')\right)(w')^2\zeta^2 +h'w'\zeta^4,
	\end{aligned}
	\end{equation}
	meaning that
	$$
	\begin{aligned}
		\left(-\theta'' -\varepsilon \frac{\theta'''\theta'-(\theta'')^2}{(\theta')^2} \right)(w')^4\zeta^4    \leq & (w')^2\zeta^3\zeta_t+ \varepsilon (2(\zeta')^2-\zeta\zeta'')(w')^2\zeta^2+|h'||w'|\zeta^4\\
		& -2\left(\varepsilon \frac{\theta''}{\theta'} + \theta'\right)(w')^3\zeta^3\zeta'.
	\end{aligned}
	$$ 
	Due to the properties of $\theta$, the factor on the left-hand side can be bounded below by a positive constant $c_1$ since $\varepsilon\leq\varepsilon_0<1$, and the factor on the right-most term can be bounded above by a positive $c_2$. We deduce that, at the maximum point, and for a constant $C$ depending on $c_2$ and $C_\varphi$,
	$$
	c_1|v|^4\leq c_2|\zeta'||v|^3+ \left(|\zeta\zeta_t|+\varepsilon |2(\zeta')^2-\zeta\zeta''|\right)  |v|^2 +\bar{h'}\zeta^3|v|\leq 2C(|v|^3+|v|^2)+\bar{h'}|v|,
	$$ 
	where
	$$
	\bar{h'}:=\max\limits_{z\in E,t\in[0,T]}|h'(z,t)|.
	$$
	Note that $\bar{h'}$ depends only on the domain $E$, but since $E$ has been chosen freely depending on $D$, it actually depends on the domain $D$, the function $R(z)$ and $\tau$.
	
	The last inequality implies that there must exist a positive constant $k(R,\tau, D)$ such that
	$$
	\sup\limits_{(t, z)\in E_T}|w'\zeta|=|w'\zeta (t_m, z_m)|\leq k(R,\tau, D).
	$$
	We next compare $w'$ and $w'\zeta$ as follows
	$$
	\sup\limits_{(t, z)\in [0,T]\times D}|w'|= \sup\limits_{(t, z)\in [0,T]\times D}|w'\zeta|\leq \sup\limits_{z\in E_T}|w'\zeta|\leq k(R,\tau, D),
	$$
	and thus, since $u_\varepsilon'=\theta'w'$ and $|\theta'|\leq 1$,
	$$
	|u'_\varepsilon(t,z)|\leq k(R,\tau, D) \text{ for all }(z,t)\in [0,T]\times D.
	$$
	
	We conclude  by taking as our desired bound the maximum between this bound $k$ and $k_D$ from~\eqref{eq:k_tilde}.
	
	\noindent\textbf{Proof of Proposition~\ref{prp:regularity} (\textit{iii})}:
	
	This part can be proved similarly to the proof in Step 5 of Section 2.1, where the Bernstein Method is applied, but in the specific case of one dimension.

    \noindent\textbf{Proof of Proposition~\ref{prp:regularity} (\textit{iv})}:

    Once the locally uniform bounds   and the uniform Lipschitz bound with respect to $z$ are obtained for $u_\varepsilon$, then the equi-continuity in time follows  using standard arguments, see~\cite{BarlesBitonLey}.
    
\end{proof}
	
\subsection{Convergence to the Hamilton-Jacobi equation}	
	At this point, even though we obtained regularity estimates for $u_\varepsilon$ that assure its convergence, as $\varepsilon\to 0$ and along subsequences, to a continuous function $u$, we do not have enough regularity in time of the limits $\rho$ and $\Phi$ in order to pass to the limit in $\varepsilon$ directly in equation~\eqref{eq:hopf-cole}. To by-pass this complication we define the auxiliary function
	\[
	w_\varepsilon(t,z):=u_\varepsilon(t,z)+\int_0^t\rho_\varepsilon(s)\ \rm{d}s-\int_0^t\Phi_\varepsilon(s, z)\ \rm{d}s.
	\]
	Notice that $w_{\varepsilon,0}(z):=w_\varepsilon(0, z)=u_{\varepsilon, 0}(z)$. With our previous estimates, we can pass to the limit in the equation satisfied by $w_\varepsilon$,
	\[
	\partial_t w_\varepsilon(t,z) =\varepsilon \partial_{zz}^2 w_\varepsilon(t,z)+\left| \partial_z w_\varepsilon(t,z) +\int_0^t\Phi'_\varepsilon(s, z)\ \rm{d}s\right|^2 + R(z) + \varepsilon\int_0^t\Phi''_\varepsilon(s, z)\ \rm{d}s,
	\]
	to obtain the existence of a continuous function $w$ that is a viscosity solution of
	\[
	\partial_t w(t,z) =\left| \partial_z w(t,z) +\int_0^t\Phi'(s, z)\ \rm{d}s\right|^2 + R(z)
	\]
	with initial datum $w_0=u_0$. Notice that using the notion of viscosity solutions for Hamiltonians with $L^1$ dependence with respect to $t$ \cite{HI:85,PL.BP:87}, this also implies that $u$ is a viscosity solution to \eqref{eq:viscosity_u}. Following the Dynamic Programming Principle, see~\cite{BardiCapuzzoDolcetta-1997,Lions-1982}, a variational solution of the previous equation is given by the formula
	\[
	w(t,z) = \sup_{(\gamma(s),s)\in\R^d\times[0,t], \gamma(t)=z}\Big\{w_0(\gamma(0))+\int_0^t L(s,\dot\gamma, \gamma)+R(\gamma(s))\ \rm{d}s\Big\},
\]
where $\gamma\in W^{1,2}([0,t] : \R)$, with
\[
L(t,q,z):=\inf\limits_{p \in \R}\left\{q\cdot p+\left|p + \int_0^t\Phi'(s, z)\ \rm{d}s\right|^2  \right\}.
\]
This infimum of the functional above is attained at the value $p=-q/2 - \int_0^t\Phi'$ and thus
\[
L(t,q,z)=-\frac{q^2}{4} - q\int_0^t\Phi'(r, z)\ \rm{d}r
\]
so
\[
w(t,z)=\sup\left\{u_0(\gamma(0))+\int_0^t-\frac{\dot\gamma(s)^2}{4}-\dot\gamma(s)\int_0^s\Phi'(r, \gamma(s))\ \rm{d}r + R(\gamma(s))  \rm{d}s \right\}.
\]
However, exchanging the order of integration we have that
\[
\begin{aligned}
\int_0^t \int_0^s -\dot\gamma(s)\Phi'(r, \gamma(s))\ \rm{d}r\ \rm{d}s& = \int_0^t \int_r^t -\dot\gamma(s)\Phi'(r, \gamma(s))\ \rm{d}s\ \rm{d}r \\
&= \int_0^t -\Phi(r,\gamma(t))+\Phi(r, \gamma(r))  \ \rm{d}r.
\end{aligned}
\]
Using that $\gamma(t)=z$ and that
\[
	w(t,z)=u(t,z)+\int_0^t\rho(s)\ \rm{d}s-\int_0^t\Phi(s, z)\ \rm{d}s,
\]
we obtain the formula given in Theorem~\ref{thm:limit_epsilon}.

\subsection{The inequality $u\leq 0$ and the convergence of $n$}\label{sec:constrain_u_and_convergence_n}

We first notice that $n_\e$ is uniformly bounded  in $L^\infty(\R^+ ;L^1(\R))$. This implies that, as $\e\to 0$ and along subsequences, $n_\e$ converges in $L^\infty(w \ast(\R^+);\mathcal M^1(\R))$ to a measure $n\in L^\infty(\R^+;\mathcal M^1(\R))$. Moreover, $\rho_\e$ is uniformly bounded in $L^\infty(\R^+)$. Hence, it converges as $\e\to 0$ and along subsequences, in $L^\infty(w\ast (\R^+))$ to a function $\rho(t)\in L^\infty(\R^+)$. Furthermore, $\Phi_\e$ is uniformly bounded in $L^\infty(\R^+;C^3(\R)$. Therefore, it converges, along subsequences, in  $L^\infty(w\ast(\R^+); C^2 (\R))$ to a function $\Phi\in L^\infty(\R^+;C^2( \R))$. Moreover, from the bounds in \eqref{bound-rho-e}, we obtain the bounds in \eqref{bound-rho}.

We next obtain the bound $u\leq 0$. Let $u$ be the limit of $u_\e$ along a subsequence, that we will fix and consider our arguments along this subsequence.  We argue by contradiction. Suppose that at a certain point $z_1$  we have $u_\varepsilon(t,z_1)\to u(t,z_1)>0$. Then by the equicontinuity of the family $u_\varepsilon$, there exists  $\delta>0$ such that $u(t,z)>0$ for all $z\in(z_1-\delta, z_1+\delta)$. We next notice that
\begin{equation}
\label{rho-u}
\rho_\e(t)=\int_\R e^{\frac{u_\e(t,z)}{\e}}dz.
\end{equation}
The Fatou's Lemma would then imply that $\rho(t)=\infty$, which is a contradiction with the upper bound on $\rho(t)$ given in \eqref{bound-rho}.

Finally, let us prove the second statement in~\eqref{eq:support}. Suppose that $z_1\not\in \{z\in\R\ :\ u(t,z)=0\}$. Then, by the inequality $u\leq 0$ and the continuity of $u$ we deduce that there exists an open set $D$ such that $z_1\in D$  and
\[
\lim\limits_{\varepsilon\to 0} u_\varepsilon(t,z)<0\quad\text{for all}\quad z\in D.
\]
Choose now any $f\in C_c(\R)$ such that $\supp (f)\subseteq D$. Then, 
\[
\int_\R n(t,z)f(z)\ \d z =  \int_\R \lim\limits_{\varepsilon\to 0}  \left(e^{\frac{u_\varepsilon(t,z)}{\varepsilon}}\right)f(z)\ \d z =0.
\]
This implies that $n(t,z)\equiv 0$ for all $z\in D$ and, in particular, $z_1\not\in \supp(n(t,\cdot))$, proving our claim.

 \section{Proof of Theorem~\ref{thm:limit_epsilon_non_concave}.}\label{sect:non_concave}
	
	In this section we will use similar  ideas from  Section~\ref{sec:Dyn_Prog} to obtain the result, but the adaptation of the arguments to our problem requires some work. In particular, given a general $T>0$, we may not  necessarily have a set $\Omega_0$ where the fitness function $F$ is positive, greater than its values outside $\Omega_0$ and concave in $[0,T]\times \Omega_0$.
	
	The form of $u_0$ and the concavity of the fitness function around the value $z_0$ can indeed assure that the solution $u$ remains concave   for small times, but as time goes by $u$ may lose its concavity. New local maximum points can appear that may reach the value $u=0$. We will show that, as long as $\mu\in [0,\mu_1]$ and $z_0\leq \mu$, this cannot be the case, and while $u$ may become non-concave in a global sense, it will still have a single global maximum corresponding to a continuous monomorphic situation. We recall the definition of
	\[
	\z(t)\quad\text{as the unique point where}\quad \max\limits_{z\in\R}u(t,z)=u(t,\z(t))
	\]
	whenever this value is unique and well defined. Let us discuss briefly the structure of this section.
	
	In Section~\ref{sect:general_considerations} we prove that $u$ remains monomorphic at least for some small times $t\in[0, \varepsilon]$. First, we will see that it remains uniformly concave after the initial time, then we will characterize the maximum point with an equation of the form~\eqref{can-eq}, and then we will prove that, at least for small times, $\rho(t)>0$, which will lead to the constraint $\max u=0$. All this imply the existence of a positive time $T_m$ until which the solution is continuous monomorphic, and the next two sections are devoted to proving that $T_m=\infty$. Some more properties of the fitness function $F$ are also studied at the end of this subsection.
	
	Section~\ref{sec:z_0=mu} treats the simpler case where $\z(0)=\mu\leq\mu_1$, and Section~\ref{sect:z_0<mu} with the case $\z(0)<\mu\leq \mu_1$, which is more convoluted; we will integrate both analytic and geometric approaches to extend the ideas from Section 2 to this case, though with significant modifications. This is where the  analysis of the problem becomes more intricate and demanding.

	\subsection{Structure of the Fitness Function}\label{sect:general_considerations}
	
	Let us now study the general structure of the fitness function. Its formula, even if the solution $u$ is not monomorphic, is given by
	\[
	F(t, z)=R(z)- \rho(t)+\tau\int_\R \frac{n(t,y)}{\rho(t)} H(z-y)\ {\rm d}y .
	\]
	Thanks to the regularity properties of the kernel $H$ this function is differentiable in the $z$ variable, so we can compute
	\[
	\partial_z F(t, z)=R'(z)+\tau\int_\R \frac{n(t,y)}{\rho(t)}H'(z-y)\ {\rm d}y.
	\]
	Clearly, again by the regularity of $R$ and $H$, this function is continuous in the $z$ variable. On the other hand, while $F$ may be discontinuous in time, it is bounded by
	\[
	R(z)-\rho_{max}-\tau\leq F(t,z)\leq 1+\tau.
	\]
	This is enough to adapt most of the arguments of Section~\ref{sec:Dyn_Prog_1}.

	\begin{lem}\label{lem:deriv_F_cont}
		Let conditions of Theorem~\ref{thm:limit_epsilon_non_concave} hold. If  at a time $t_0\geq 0$ the solution $u$ has a unique zero $\z(t_0)$, then for all $z\in\R$ the functions $\partial_z F(t,z)$ and $\partial_{zz}^2 F(t,z)$ are continuous in time, uniformly with respect to $z$, in a neighbourhood of $t_0$.
	\end{lem}
	\begin{proof}
		Recall that, by proposition~\ref{prp:regularity}, $u$ is continuous in time. Fix $z\in\R$ and choose now any $\sigma>0$. Since the measure $n$ is supported in the set of zeroes of $u$, and thanks  to the continuity of the solution $u$ and the uniform bounds \eqref{uniform-bound-u}, we have that, given this $\sigma$, there must exists a $\delta>0$ such that, if $|t_0-t|\leq 2\delta$, then
		\[
		\supp(n(t,\cdot))\subseteq [\z(t_0)-\sigma, \z(t_0)+\sigma].
		\]
		This means that given any $\sigma>0$ there exists a $\delta>0$ such that, if $|t_0-t|\leq 2\delta$ and $|t_0-s|\leq 2\delta$, then
		\[
		\begin{aligned}
			&\left| \partial_z   F(t, z) - \partial_z F(s, z)  \right|\leq \tau \left| \int_{\z(t_0)-\sigma}^{\z(t_0)+\sigma}\frac{n(t,y)}{\rho(t)} H'(z-y) \ {\rm d}y- H'\left(z-\z(t_0)\right)  \right|\\
			&+\tau \left| \int_{\z(t_0)-\sigma}^{\z(t_0)+\sigma}\frac{n(s,y)}{\rho(s)} H'(z-y) \ {\rm d}y- H'\left(z-\z(t_0)\right)  \right|\\
			&\leq2 \tau\sup\limits_{\eta\in [-\sigma,\sigma]} \left| H'(z-\z(t_0)+\eta)-H'\left(z-\z(t_0)\right)\right|\leq {C}\tau\sigma,
		\end{aligned}
		\]
		since $\frac{n(t,\cdot)}{\rho(t)}$ integrates to 1 and $H'$ is a  Lipschitz continuous function. We deduce that the function $\partial_z F$ is continuous in time in a neighbourhood of $t_0$, at any point $z\in\R$. The continuity of $\partial_{zz}^2 F$ can be proven following similar arguments.
	
\end{proof}
	
	Having established the previous lemma, let us define
	\begin{equation}\label{def:Tc}
	\begin{array}{c}
	T_c:=\sup\Big\{t>0 : \text{for all $s\in [0,t]$, $u(s,z)$ has a unique maximum point at $\z(s)$},\\[10pt]
	\text{$\partial_{zz}^2 u(s,\z(s))<0$, $\z(\cdot)\in C^1([0,t])$,  and	 
	 $\z'(t)=\frac{2g}{|\partial_{zz}^2 u(t,\bar{z}(t))|}(\mu-\z(t))$ in $[0,t]$}\Big\}.
	\end{array}
	\end{equation}
	 We next prove that $T_c$ is greater than $0$.
	\begin{lem}\label{lem:T_m>0}
		Let conditions of Theorem~\ref{thm:limit_epsilon_non_concave} hold. Then {$T_c>0$}.
	\end{lem}
	\begin{proof}
		The function $u$ is continuous in time, implying that in small times no new maximum points can appear at a positive distance of $z_0$ thanks to~\eqref{eq:hypothesis_u_0}. However, it could happen that immediately after the initial time, the maximum point at $z_0$ divides into several, or that the maximum point extends itself into a flat interval. In both cases the number of maximum points of $u$ would augment in a neighbourhood of the point $(t,z)=(0,z_0)$. Our goal is to show that this cannot be the case and we will do so by following the same ideas from Steps 1, 2 and 4 of Section~\ref{sec:Dyn_Prog_1} in order to conclude that $u$ is strictly concave.
		
		At the initial time, due to~\eqref{eq:hypothesis_u_0}, $u_0$ has a unique maximum point at $z_0$ and thus
		\[
			\partial_{zz}^2F(0,z)=-2g+\tau H''(z-z_0).
		\]
		Moreover, by Lemma~\ref{lem:deriv_F_cont} we notice that, even though $F$ may be discontinuous in time, there   exists a time $t_0>0$ small such that $\partial_{zz}^2 F$ is continuous for all $(t,z)\in [0,t_0)\times \R$. 
		Now we note that, by hypothesis~\eqref{eq:hypothesis_H} and due to the fact that $\partial_{zz}^2F(0,z_0)=-2g<0$, there exists a point  $z_g<z_0$ such that $\partial_{zz}^2F(0,z)<-\f 32g$ for all $z>z_g$. Since $\partial_{zz}^2F(t,z)$ is continuous for all $(t,z)\in [0,t_0)\times \R$ we conclude that there   exists a time $t_1<t_0$ and a continuous curve $z_g(t)$ such that
		\[
		z_g(0)=z_g,\ z_g(t)<z_0\text{ for all }t\in[0,t_1]\ \text{and}\ \partial_{zz}^2 F(t,z)<-g \text{ for all }z>z_g(t), t\in[0,t_1].
		\]
		and, moreover, since $u$ is locally equi-continuous and initially strictly concave, we can choose $t_1$ small enough such that
		\[
		z_g(t)<\z_i(t) :=\inf\{z\in\R :  u(t,z)=\max\limits_{w\in\R} u(t,w)\}\quad\text{for all}\quad t\in[0,t_1],
		\]
		i.e., $\z_i$ is the smallest maximum point.	For convenience, let us define	the set
		\[
		\mathcal{S}(t_1,g):=\{(s,z)\in [0,t_1]\times \R :  z>z_g(s)\},
		\]
		The strict concavity of $u_0$ and $F$ in the set $\mathcal{S}(t_1, g)$   will indeed lead to the strict concavity of $u$ in a certain set.
		
		Now we adapt the arguments of  Section~\ref{sec:Dyn_Prog_1} to prove the strict concavity of $u$ in a certain set. We first notice   that 	$F$ is continuous in the $z$ variable and bounded from above by $1 +\tau$. Therefore, we can apply similar arguments  from Step 1 of Section~\ref{sec:Dyn_Prog_1} to prove that, given $(t,z)\in\mathcal{S}(t_1,g)$ there exists a trajectory $\gamma_z$ such that $\gamma_z(t)=z$ and $u(t,z)=f_t(\gamma_z)$, where $f_t$ comes from~\eqref{eq:dyn_prog}. Moreover, this trajectory satisfies the system~\eqref{eq:EL} from Step 2 of Section~\ref{sec:Dyn_Prog_1}, which is coercive whenever $F$ is strictly concave, and thus $\gamma_z$ will be unique if $z>z_g(t)$, $t\leq t_1$, as long as $\gamma_z(s)\geq z_g(s)$ for any $s\in[0,t]$. The idea now is to follow Step 4 of Section~\ref{sec:Dyn_Prog_1} to prove strict concavity, but in order to apply the same ideas we have to ensure that $\gamma_z(s)>z_g(s)$ for all $s\in[0,t]$. To prove this property, we cannot rely entirely on the ideas presented at Step 2 of Section~\ref{sec:Dyn_Prog_1}, since~\eqref{hyp2} is not satisfied in general. We will obtain the desired property by constructing suitable left barriers for these trajectories instead.
		
		Again, since $z_g(0)<\z_i(0)$ and both curves $z_g$ and $\z_i$ are continuous at $t=0$, we can choose $t_1$ small enough such that
		\begin{equation}\label{eq:distance_z_g_z_i}
			\inf\limits_{0\leq s\leq t_1} \z_i(s) - \sup\limits_{0\leq s\leq t_1} z_g(s) > 2C\sqrt{t_1},
		\end{equation}
		with $C$ the constant given in \eqref{eq:bound_size_trajectory}.
		In particular, if we define the intermediate point
		\[
		z_m:=\frac{1}{2}\cdot \left(\sup\limits_{0\leq s\leq t_1} z_g(s) + \inf\limits_{0\leq s\leq t_1} \z_i(s)\right),
		\]
		then, given any time $t\in[0,t_1]$, any trajectory $\gamma_{z_m}$ such that $\gamma_{z_m}(t)=z_m$, satisfies, due to~\eqref{eq:bound_size_trajectory} and~\eqref{eq:distance_z_g_z_i}, that
		\[
		\sup\limits_{0\leq r\leq t} z_g(r)<\gamma_{z_m}(s)<\inf\limits_{0\leq r\leq t} \z_i(r)\quad\text{for all}\quad s\in[0,t].
		\]
		As a consequence, $(\cdot,\gamma_{z_m}(\cdot))\subset \mathcal{S}(t_1,g)$ and it is the unique optimal trajectory that passes through its points. Therefore, if $z>z_m$, then $\gamma_z(s)\geq\gamma_{z_m}(s)$ for all $s\in[0,t],\ t\leq t_1$. In particular, ${(\cdot,\gamma_z(\cdot))}\subset \mathcal{S}(t_1,g)$ and thus, following Step 4 of Section~\ref{sec:Dyn_Prog_1}, $u(t,z)$ is strictly concave in the set $[0,t_1]\times [z_m,\infty)$, with $\partial_{zz}^2 u(t,z)\leq -C<0$ in $[0,t_1)\times [z_m,+\infty)$. Moreover, the continuity and the strict concavity of the solution imply that $u(t,\cdot)$ has a unique maximum point $\z(t)$ for all $t\in[0,t_1]$.
		
		By adapting the ideas in Step 6 of Section~\ref{sec:Dyn_Prog_1}, see also~\cite{Mirrahimi-Roquejoffre - 2016}, we can derive the formula~\eqref{can-eq} and prove that 
	\begin{equation}\label{eq:z_is_Lipschitz}
		\bar{z}\in C^{0,1}([0,T_c)),\text{ with }\begin{cases} \displaystyle\bar{z}'(t)=-\frac{\partial_{z} F(t,\z(t))}{\partial_{zz}^2 u(t,\bar{z}(t))} = \frac{2g}{|\partial_{zz}^2 u(t,\bar{z}(t))|}(\mu-\z(t)),\\
		\z(0)=z_0.
		\end{cases}
	\end{equation}
	We omit the details for the sake of brevity, but we highlight that one do not need to assume $\max u=0$ (which is a condition appearing in problem~\eqref{eq:1}) in order to develop the analysis leading up to formulas~\eqref{can-eq} or~\eqref{eq:z_is_Lipschitz}. We hence conclude that $T_c\geq t_1>0$.
	\end{proof}
	
	Notice from   the ODE in~\eqref{eq:z_is_Lipschitz} that
	\[
	\z(0)<\z(t)<\mu\quad\text{and}\quad \z'(t)>0\quad\text{ for all }t\in(0,T_c).
	\]
	Moreover, from~\eqref{eq:support}, we conclude that, for all $t\in[0,T_c)$,
	\[
	n(t,z)=\rho(t)\cdot \delta_{\bar{z}(t)}(z)\quad\text{with}\quad \rho(t)\in[0,\rho_{max}].
	\]
	The next step is to characterize the quantity $\rho(t)$.

	\begin{lem}
	\label{lem:valrho}
	Let conditions of Theorem~\ref{thm:limit_epsilon_non_concave} hold. Then  $\rho_\e$ converges pointwise in $(0,T_c)$ to $\rho(t)=\max(R(\bar z(t)),0)=\max(1-g\bar z(t)^2,0)$. Moreover, as long as $\rho(t)>0$, we have $\max\limits_{z\in\R}u(t,z)=u(t,\z(t))=0$.
	\end{lem}
	\begin{proof}
		
	Let $t\in(0,T_c)$. Integrating in the whole $\R$ in~\eqref{eq:main} and by arguments similar to the ones in Lemma~\ref{lem:rho_phi_weak_convergence}, we have
	\[
	\varepsilon \rho_\varepsilon'(t)=-\rho_\varepsilon(t)^2 +\int_\R n_\varepsilon(t,z)R(z){\rm dz}.
	\]
	From the convergence of $n_\varepsilon$ we deduce that, for $\varepsilon$ small enough,
	\[
	\varepsilon \rho_\varepsilon'(t)=-\rho_\varepsilon(t)^2 +\rho_\varepsilon(t)(R(\z(t))+o(1))
	\]
	and we recall that $0\leq \rho_\varepsilon\leq \rho_{max}$
	
	We define now the auxiliary quantity
	\[
	J_\varepsilon(t):=\ln(\rho_\varepsilon(t)).
	\]
	Replacing this in the previous equation, we obtain that
	\[
	\varepsilon J'_\varepsilon(t)=R(\bar{z}(t))-e^{J_\varepsilon(t)}+o(1).
	\]
	
	We study the behaviour of $\rho_\varepsilon$ by studying the one of $J_\varepsilon$. To this end, we define the function $\mathcal{J}(t,z)$ as the unique solution of the ODE 
		\[
	\begin{cases}
		\partial_s \mathcal{J}(s,z)=R(z)-e^{\mathcal{J}(s,z)},\quad&\text{for all }z\in D_R,\ s\geq 0\\
		\mathcal{J}(0,z)=\mathcal{J}_0 &\text{for all } z\in D_R.
	\end{cases}
	\]
	A solution of such ODE is of the form
	\[
	\mathcal{J}(s,z)= \ln(R(z)) + R(z)s-\ln\left(R(z)e^{-\mathcal{J}_0} + e^{R(z)s}-1\right)
	\]
	
	and, moreover,  
	\begin{equation}
	\label{limJ}
	\begin{cases}
	\mathcal{J}(s,z)\to \ln(R(z))&\text{as}\; s\to \infty \text{ and if $R(z)>0$},\\
	\mathcal{J}(s,z)\to -\infty&\text{as}\; s\to \infty \text{ and if $R(z)\leq 0$}.
	\end{cases}
	\end{equation}
	The proof of the previous claim is rather simple and we omit it for the sake of brevity. We are ready now to compute
	\[
	\varepsilon\cdot\partial_t\left[\mathcal{J}\left(\frac{t}{\varepsilon},\bar{z}(t)\right) - J_\varepsilon(t)\right]=e^{J_\varepsilon(t)}-e^{\mathcal{J}\left(\frac{t}{\varepsilon},\bar{z_\varepsilon}(t)\right)}+\varepsilon\bar{z}'(t)\cdot \partial_z \mathcal{J}\left(\frac{t}{\varepsilon},\bar{z}(t)\right)+o(1),
	\]
	but, by~\eqref{eq:z_is_Lipschitz}, the term	 $\bar{z}'(t)$ is uniformly bounded, meaning that we have, in fact,
	\[
	\varepsilon\cdot\partial_t\left[\mathcal{J}\left(\frac{t}{\varepsilon},\bar{z}(t)\right) - J_\varepsilon(t)\right]=e^{J_\varepsilon(t)}-e^{\mathcal{J}\left(\frac{t}{\varepsilon},\bar{z}(t)\right)}+o(1).
	\]
	Multiplying by the sign of the term between brackets we obtain
	\[
	\begin{aligned}
		\varepsilon\cdot\partial_t\left|\mathcal{J}\left(\frac{t}{\varepsilon},\bar{z}(t)\right) - J_\varepsilon(t)\right|&=-\left|e^{J_\varepsilon(t)}-e^{\mathcal{J}\left(\frac{t}{\varepsilon},\bar{z}(t)\right)}\right|+o(1)\\
		&\leq -k\left|  \mathcal{J}\left(\frac{t}{\varepsilon},\bar{z}(t)\right)-J_\varepsilon(t)\right|+o(1),
	\end{aligned}
	\]
	for a certain positive constant $k$. Notice that using the bounds for $\rho_\varepsilon$ which lead to $J_\varepsilon\in (-\infty, \ln(\rho_{max})]$, we are allowed to use the Lipschitz nature of the exponential function.
	From this last inequality and the definition of $J_\varepsilon(t)$ we deduce that
	\[
	\left|\ln(\rho_\varepsilon(t))-\mathcal{J}\left(\frac{t}{\varepsilon},\bar{z}(t)\right) \right|\leq \left| \ln(\rho_\varepsilon(0))-\mathcal{J}\left(0,\bar{z}(0)\right)\right|e^{-\frac{kt}{\varepsilon}}+o(1).
	\]
	Combining this inequality with \eqref{limJ}, we conclude that
	\[
	\rho_\varepsilon(t)\to \rho(t):=\max(R(\bar{z}(t)),0)\quad\text{for all }t\in(0,T_c),\quad\text{as }\varepsilon\to 0.
	\]
		We hence proved that $\rho(t)$ is a continuous function in $[0,T_c)$. From~\eqref{eq:hypothesis_rho} and $\z(0)=z_0\in D_R$ we deduce that, at least for small times, $0<\rho(t)\leq \rho_{max}$. We then pass to the limit $\e\to 0$ in the  the relation~\eqref{rho-u} and use the upper bound in \eqref{uniform-bound-u} to deduce that
	\[
	\max\limits_{z\in\R}u(t,z)=u(t,\z(t))=0.
	\] 
	\end{proof}
	
 We can therefore define  
	\[
	\begin{array}{c}
	T_m:=\sup\Big\{t>0 : \text{for all $s\in [0,t]$, $u(s,z)$ has a unique maximum point at $\z(s)$,}\\[10pt]
		\text{$u(s,\z(s))=0$, $\partial_{zz}^2 u(s,\z(s))<0$, $\z(\cdot)\in C^1([0,t])$,  and}\\[10pt]	 
	 \z'(t)=\frac{2g}{|\partial_{zz}^2 u(t,\bar{z}(t))|}(\mu-\z(t)) \text{ in } [0,t]\Big\}.
	\end{array}
	\]
	and deduce that $T_m>0$. For $t\in [0,T_m]$, $u(t,\cdot)$ is locally strictly concave with respect to $z$ in a neighbourhood of its maximum point located at the continuous curve $\z(t)$, satisfying the constraint $\max u(t,z)=0$. Moreover, since $\z(t)$ is the unique maximum point of $u(t,\cdot)$, for all $t\in [0,T_m)$, and ${\rm supp} \;n(t,\cdot)\subset \{z :  u(t,z)=\max_{y\in\R} u(t,y)=0\}=\z(t)$, we obtain that $n(t,z)$ must be a Dirac delta with mass $\rho(t)$ located at the point $\z(t)$, and therefore $\Phi(t,z)=\tau H(z-\z(t))$. We deduce that, for $t\in[0,T_m)$,
$$
\partial_t u=(\partial_z u)^2 + F(t,z),
$$
with 
\[
\label{FR=z2}
F(t,z)=-g(z^2-\bar{z}(t)^2)+\tau H(z-\bar{z}(t)).
\]
Note that thanks to the continuity of $\z$ in $ [0,T_m)$, $F(t,z)$ is continuous with respect to $t$  for $t\in  [0,T_m)$.
We then focus on the following problem
		\begin{equation}\label{eq:v_ham_jac}
		\begin{cases}
			\partial_t u(t,z)=(\partial_z u(t,z))^2 -g(z^2-\bar{z}(t)^2)+\tau H(z-\bar{z}(t)),\quad &t\in[0,T_m), z\in\R,\\
			\sup\limits_{z\in \R} u(t,z)=u(t,\bar{z}(t))=0, &t\in[0,T_m),\\
			u(0,z)=u_0(z), &z\in\R.
		\end{cases}
	\end{equation}
	We will prove that, if for all $t\in [0,T_m]$, $\rho(t)=R(\z(t))>0$, then $T_m=\infty$. We will do so by assuming that $T_m$ is finite and proving that, indeed, there are posterior times, after $T_m$, where $\bar{z}(t)$ is well-defined, unique and continuous with respect to $t$, which is in contradiction with the definition of $T_m$. Also, since for these times the curve $\z(t)$ is unique, we will write $\z(0)$ or $z_0$ indistinctively. We prove below, under a strict concavity assumption at $T_m$, that the only possibility for loosing monomorphism is that new zeroes of $u$ appear at time $T_m$ far away from the old one.
	Note that by equation~\eqref{eq:z_is_Lipschitz} we have that the value
	\[
	\z(T_m):=\lim\limits_{t\nearrow T_m} \z(t)
	\]
	is well defined.

\begin{lem}\label{lem:T_m=infty}
	Let conditions of Theorem~\ref{thm:limit_epsilon_non_concave} hold and assume that  $T_m<\infty$, that $\partial_{zz}^2u(T_m,\cdot)\leq -C<0$ in $(z_m,+\infty)$, for some $z_m<\z(T_m)$, and that for all $t\in [0,T_m]$, $\rho(t)=R(\z(t))>0$.
	Then, $u(T_m,\cdot)$ has a new zero point at a positive distance of $\z(T_m)$.
\end{lem}
\begin{proof}
	We will argue by contradiction and in a similar way to Lemma~\ref{lem:T_m>0}. Suppose that $T_m<\infty$ and that no new zeroes appear at a positive distance from $\z(T_m)$ at time $T_m$.
	
	In this case, at time $T_m$ the value $\z(T_m)$ is the only zero point of $u$, implying that
	\[
	\partial_{zz}^2F(T_m,z)=-2g+\tau H''(z-\z(T_m)).
	\]
	Notice that by Lemma ~\ref{lem:deriv_F_cont} this function is continuous in time,  and hence negative, in $(T_m-\delta,T_m+\delta)\times (z_m',+\infty)$ for   some constants $z_m'<\bar z(T_m)$ and $\delta>0$.  Moreover,    $u(T_m,\cdot)$ has a single maximum point at $\z(T_m)$ and it  is strictly concave with respect to $z$ in $  (z_m,+\infty)$, with $z_m<\z(T_m)$. We can hence argument analogously to Lemma~\ref{lem:T_m>0} and prove   the required properties in the definition of $T_c$ in \eqref{def:Tc} for a small interval of time after $T_m$. Moreover, since $R(\z(T_m))>0$, and thanks to the continuity of $\z$ in such an interval, we deduce, thanks to Lemma \ref{lem:valrho}, that $\rho(t)>0$,  and hence $u(t,\z(t))=0$, in such an interval, making it smaller if necessary. This implies that $T_m$ could have been extended to higher values, which is   a contradiction.
\end{proof}

	Let us recall now some properties of the function $H$ obtained in~\cite[Section 4]{Garriz-Leculier-Mirrahimi} leading to the quantification of $\mu_1$. Hypothesis~\eqref{eq:hypothesis_H} implies, after some analysis, that
	\begin{equation}\label{eq:d1}
		\text{ there exists a unique positive value $d_1$ solving \;} d_1(1+H'(d_1))-2H(d_1)=0.
	\end{equation}
	Then, the threshold $\mu_1$ is given by
	\[
	\frac{2H(d_1)}{(1-H'(d_1))(1+H'(d_1))}=\frac{d_1}{1-H'(d_1)}=\mu_1.
	\]
	We recall that it was shown in \cite[Section 4.2]{Garriz-Leculier-Mirrahimi} that $\mu_1$ corresponds to the threshold on $\mu$ below which there exists a monomorphic stationary solution of~\eqref{eq:main}.
	Let us define the auxiliary function
		\begin{equation}\label{eq:varphi}
			\varphi(t,z):=\frac{F(t,z)}{2g},
		\end{equation}
		which has the advantage of simplifying certain computations. Now we define $x(t):=\mu-\bar{z}(t)$. Notice that $\varphi(t,\mu)=f(x(t),\mu)$ with  
		\[
		f(x,\mu)=\mu H(x)-x\left(\mu-\frac{x}{2}\right).
		\]
	\begin{lem}\label{lem:F(t,mu)_greater_0}
		Let conditions of Theorem~\ref{thm:limit_epsilon_non_concave} hold.\\
		(i) Then, $f(x,\mu)\geq 0$ for all  $\mu \leq \mu_1$ and $x\geq 0$. The equality holds only for $x=0$ or $x=d_1$ and $\mu=\mu_1$.\\
		(ii) Let $t\leq T_m$. Then $F(t,\mu)>0$ except when $\z(t)=\mu$ or $\z(t)=\mu_1-d_1$ and $\mu=\mu_1$. In these last two cases, we have $F(t,\mu)=0$. Consequently, $F(t,\mu)\geq 0$ for all $t\in[0,T_m],\  \mu \leq \mu_1$. 
	\end{lem}
	\begin{proof}
		
		We notice,  thanks to assumption \eqref{eq:hypothesis_H}, that $f(0,\mu)=0$, $\partial_x f(0,\mu)=0$ and $\partial^2_{xx} f(x,\mu)=\mu H''(x)+1>0$, for all $x>0$ and small values of $\mu$. This implies that, for all $x>0$ and for small values of $\mu$, $f(x,\mu)>0$. We define $\mu_0=\inf \{\mu>0\,|\, \exists x>0, f(x,\mu)=0\}$ which is necessarily positive in view of the previous arguments.  The following system of  equations   must then be satisfied at a point $x_0>0$
		\[
		\begin{cases}
			\displaystyle\mu_0 H(x_0) =x_0\left(\mu_0-\frac{x_0}{2}\right),\\
			\mu_0 H'(x_0)=\mu_0 -x_0.
		\end{cases}
		\]
		From the second equation we deduce that $2\mu_0 - x_0 = \mu_0(1+H'(x_0))$, and substituting this in the first equation provides $x_0(1+H'(x_0))=2H(x_0)$, {implying thanks to}~\eqref{eq:d1} that $x_0=d_1$. Since from the second equation we can deduce also that $x_0=\mu_0(1-H'(x_0))$ we deduce that
		\[
		\mu_0 =\frac{d_1}{1-H'(d_1)}=\frac{2H(d_1)}{(1-H'(d_1))(1+H'(d_1))}=\mu_1.
		\]
		It follows that $f(x,\mu)\geq 0$ for all  $\mu \leq \mu_1$ and $x\geq 0$. Moreover, the equality holds only for $x=0$ or $x=d_1$ and $\mu=\mu_1$.
		Translating this to our variables $(t, \z(t))$, and using the fact that $x(t)=\mu-\z(t)\geq 0$, we deduce that, in the range $\mu\in[0,\mu_1]$, $\varphi(t,\mu)\geq 0$, and hence $F(t,\mu)\geq 0$, and the only moment when $\varphi(t,\mu)=0$, and hence $F(t,\mu)=0$, is when $\mu=\mu_1$ and $\z(t)=\mu_1-d_1$, or when $\z(t)=\mu$.
	\end{proof}

In what follows we will focus on the case where
\[
\z(0)\in\left(-\frac{1}{\sqrt{g}},\mu\right].
\]
Since the limit case $\z(0)=\mu$ can be regarded as a  special, simpler case, we will deal with it first.

\subsection{The case $\z(0)=\mu\leq \mu_1$}\label{sec:z_0=mu}

Suppose for now $\mu<\mu_1$ and assume that $\z(0)=\mu$ and $u_0(z)<0$ for all $z\neq \mu$. Notice that Assumption \eqref{eq:hypothesis_rho} implies that $\tau<2\sqrt{g}$ and that $\rho_0=1-\f{\tau^2}{4g}>0$. 

We also notice that   $F(0,z)=g\big(- z^2+\mu^2+2\mu H(z-\mu)\big)=-2gf(\mu-z,\mu)$. Thanks to Lemma \ref{lem:F(t,mu)_greater_0} we deduce that $F(0, z)<0$, for all $z\neq \mu$.  Moreover, we have
\[
F(0,\mu)=0,\quad \partial_z F(0,\mu)=0,\quad \partial_{zz}^2F(0,\mu)=-2g.
\]

We deduce then from the previous analysis, i.e. Lemmas \ref{lem:T_m>0}  and and \ref{lem:valrho}, that $T_m>0$ and that equation~\eqref{eq:z_is_Lipschitz} is satisfied. Therefore, $\z(t)=\mu$ for all $t\in[0,T_m)$ and
\[
F(t,z)= F(0,z)\quad\text{for all }t\in[0,T_m).
\]
We next apply Theorem \ref{thm:section_Dyn_Prog_1} to obtain that no new zeroes of $u$ appear far away from the point $(T_m, \mu)$ and that $\partial_{zz}^2u(T_m, z)<0$ in a neighbourhood of the point $(T_m,\mu)$. By Lemma~\ref{lem:T_m=infty} we conclude that $T_m=\infty$.

On the other hand, if we suppose that $\mu=\mu_1$ the same argument holds true except for the fact that this time the fitness function $F$ attains two zeroes, with  the second one located at $\mu_1-d_1$, with $d_1$ given by~\eqref{eq:d1}. However, the fact that $u_0(\mu_1-d_1)<0$ ensures that $u(t,\mu-d_1)<0$ for all $t\geq 0$ again by~\eqref{eq:dyn_prog}. This is enough to conclude the results of Theorem~\ref{thm:limit_epsilon_non_concave} in the case $\z(0)=\mu\leq \mu_1$.

\subsection{The case $\z(0)<\mu\leq \mu_1$}\label{sect:z_0<mu}

Once again, our goal is to study the problem up to time $T_m$ assuming that $T_m<\infty$, and to conclude thanks to a contradiction on the definition of $T_m$ that, in fact, either $\rho(T_m)=0$ or $T_m=\infty$. In what follows we hence  assume that $T_m<\infty$. Note that, by equation~\eqref{eq:z_is_Lipschitz}, we can assure that $\z(t)<\mu$ for all $t\leq T_m$. In view of Lemma \ref{lem:T_m=infty}, in order to prove the result, it is enough to show that no new zero point of $u(T_m,\cdot)$ appears far away $\z(T_m)$ and that $\partial_{zz}^2u(T_m,\cdot)$ is strictly negative around $\z(T_m)$. In order to do so we follow the ideas presented in Section~\ref{sec:Dyn_Prog_2}. However, we cannot apply the results of Section~\ref{sec:Dyn_Prog_2} directly since this time the fitness function changes over time and all the conditions on $F$ might not be satisfied in a fixed convex set $[0,T]\times\Omega_0$. In order to circumvent this difficulty we need to study the fitness function and the evolution of the system in more detail.

\begin{lem}\label{lem:existence_curve_gamma}
	Let conditions of Theorem~\ref{thm:limit_epsilon_non_concave} hold and that $T_m<\infty$. Given $(t,z)\in [0,T_m]\times \R$, there exists a curve $\gamma_z(s)\in W^{1,2}([0,t],\R)$ such that $z=\gamma_z(t)$ and
	\[
	u(t,z)=u_0(\gamma_z(0))+\int_0^t\Big(-\frac{|\dot{\gamma_z}(s)|^2}4+F(s,\gamma_z(s))\Big)ds.
	\]
\end{lem}

The proof is similar to the one of Step 1 in Section~\ref{sec:Dyn_Prog_1}, using only that $F$ and $u_0$ are bounded from above and that $u$ is well defined and locally bounded. Note however that this trajectory may not be unique.

\begin{lem}\label{lem:prop_F}
Let conditions of Theorem~\ref{thm:limit_epsilon_non_concave} hold and $\mu< \mu_1$, $\bar{z}(0)\in\left(-1/\sqrt{g},\mu\right)$, and that $T_m<\infty$. Then
\begin{itemize}
\item[(i)] for all $t\in[0,T_m],\ z\in   (\z(t),+\infty)$, $\partial_{zz}^2F(t,z)<-2g$,
$F(t,\bar{z}(t))=0<F(t,\mu)$ and 
 $\partial_z F(t,\bar{z}(t))>0$.
\item[(ii)]  For all $0<t\leq T_m,\ z\geq\mu$, $\partial_z F(t,z)< 0$ , and the value
\[
r(t):=\{r >\z(t)\  :\ \max\limits_{z\in (\z(t), \infty)} F(z, t)= F(r,t)\}
\]
is unique and satisfies $\bar{z}(t)<r(t)<\mu$.
\end{itemize}
\end{lem}

\begin{proof}

We recall from Lemma~\ref{lem:F(t,mu)_greater_0} that $F(t,\mu)>0$ except when $\z(t)=\mu$ or $\mu=\mu_1$ and $\z(t)=\mu_1-d_1$. Recalling that $\mu>\z(t)$ for all $t\in [0,T_m]$, we deduce that $F(t,\bar{z}(t))=0<F(t,\mu)$ for all $t\in [0,T_m]$. The rest of the claims in \textit{(i)} and \textit{(ii)} come after a brief analysis of $F$. 

Recall the definition of $\varphi$ in~\eqref{eq:varphi}. One can next verify following simple computations that $\partial_{zz}^2\varphi(t,z)=-1+\mu H''(z-\overline z(t))< -1$ for all $z> \z(t)$ and
\[
\partial_{z}\varphi(t, \z(t)) = \mu-\z(t)>0,\qquad \text{for all $t\in [0,T_m)$ and $\mu\in [0,\mu_1)$}.
\]
This provides \textit{(i)}.

To prove \textit{(ii)} we compute
\[
\partial_{z}\varphi(t, \mu) = \mu \left(H'(\mu-\z(t))-1\right)<0
\]
since $H'(x)<1$ for all $x\neq 0$, and since $\mu-\z(t)>0$. The concavity of $\varphi$ for all $z>\z(t)$ provides $\partial_z \varphi(t,z)<0$ for all $t>0, z\geq\mu$. Since $\partial_z \varphi(t,\z(t))>0>\partial_z \varphi(t,\mu)$, again from the concavity of $\varphi$ we deduce the properties of $r(t)$.

\end{proof}

{The} next result is fundamental in the study of the behaviour of the solution in the monomorphic regime $\mu\leq \mu_1$, since it proves that $F(t,z)>0$ only if $z>\z(t)$, implying, as we will prove shortly after, that no new zeroes of $u$ can appear to the left of $\z(t)$.

\begin{lem}\label{lem:F_negative}
Let conditions of Theorem~\ref{thm:limit_epsilon_non_concave} hold and that $T_m<\infty$. Let $\mu\leq \mu_1$ and  $\bar{z}(0)\in\left(-1/\sqrt{g},\mu\right)$.  Then $F(t,z)<F(t,\bar{z}(t))=0$ for all $t\in[0, T_m\SMM{]}$, $z<\bar{z}(t)$.
\end{lem}
\begin{proof}
Since the function $\z$ is continuous, then $F$, which only depends on time via $\z$, is also continuous in time and space. Moreover, since by equation~\eqref{eq:z_is_Lipschitz} $\z$ is  increasing towards $\mu$ and $\z(t)\in [-1/\sqrt{g},\mu)$ for all $t\in[0,T_m]$ we can express this dependence of $F$ on $\z$ as a dependence on a parameter $s=\z(t)$ with $s\in[-1/\sqrt{g},\mu)$ and $t=\z^{-1}(s)$, where the exponent $-1$ represents the inverse function. Let us consider then
\[
\varphi(s,z):=\frac{F(\z^{-1}(s), z)}{2g}=-\frac{z^2-s^2}{2} + \mu H(z-s).
\]
We notice that, since $\varphi(-1/\sqrt{g},-1/\sqrt{g})=0$ and $\partial_z\varphi(-1/\sqrt{g},z)>0$ for every $z<-1/\sqrt{g}$, we obtain that
\begin{equation}\label{eq:varphi_s_z}
\varphi(-1/\sqrt{g},z)<0 \text{ for all }z<-1/\sqrt{g}.
\end{equation}
Notice also that $\partial_z\varphi(s,s)=\mu-s>0$ for all $s<\mu$.

Let us argue now by contradiction in order to prove the lemma. We assume that at some $s'\in[-1/\sqrt{g},\mu)$ and $z'<s'$ we have $\varphi(s',z')\geq 0$. Increasing continuously $s$ from $-1/\sqrt{g}$ up to $s'$ we deduce, from~\eqref{eq:varphi_s_z} and the regularity of $\varphi$, that there must exist a previous $s_1\in (-1/\sqrt{g},s']$ and a point $z_1\leq s_1$ such that $\varphi(s_1, z_1)=0$ and $\partial_z \varphi(s_1,z_1)=0$. It could happen that $z_1=s_1$, but since $\partial_z\varphi(s_1,z_1)=\partial_z\varphi(s_1,s_1)=\mu-s_1>0$, we deduce that $z_1<s_1$.
In other words, the system
\[
\begin{cases}
-\frac{z_1^2-s_1^2}{2}+\mu H(z_1-s_1)=0,\\
z_1=\mu H'(z_1-s_1)=\mu H'(s_1-z_1)
\end{cases}
\]
must be satisfied.

We define $d:= s_1-z_1$. Then by the second equation of the previous system, since $s_1<\mu$, we have
\[
d=s_1-z_1<\mu (1-H'(d))\quad\text{and}\quad z_1+s_1<\mu (1+H'(d)).
\]
Since $s_1>z_1$, this implies that
\[
-\frac{z_1^2-s_1^2}{2} = \frac{(s_1-z_1)(s_1+z_1)}{2}<\mu^2\frac{(1-H'(d))(1+H'(d))}{2}.
\]
Substituting in the first equation of the system, rearranging and using that $H(-d)=-H(d)$ we obtain
\begin{equation}\label{eq:mu_greater_d}
\mu>\frac{2H(d)}{(1+H'(d))\cdot(1-H'(d))}.
\end{equation}

On the other hand, since $z_1+s_1<\mu (1+H'(d))$, substituting once again in the first equation we find
\[
\mu\cdot\big(d(1+H'(d))-2H(d) \big)>0.
\]
Following the analysis from~\cite[Lemma 4.2]{Garriz-Leculier-Mirrahimi}, this implies that $d>d_1$.
Let us now check, thanks to hypothesis~\eqref{eq:hypothesis_H}, that
\begin{equation}\label{eq:monotone_increasing}
\frac{2H(x)}{(1+H'(x))\cdot(1-H'(x))}\quad\text{is monotone increasing for all }x\geq d_1.
\end{equation}
We compute
\[
\partial_x\left[\frac{2H(x)}{(1+H'(x))\cdot(1-H'(x))}\right]=\underbrace{\frac{2H'(x)}{\big(1-(H'(x))^2)^2}}_{>0}\cdot\phi(x),
\]
where we defined
\[
\phi(x):=1-(H'(x))^2+2H(x)H''(x).
\]
We have to check that $\phi(x)>0$ for all $x>d_1$.

Since $d_1$ is the unique positive root of $x(1+H'(x))-2H(x)$, with some analysis and hypothesis~\eqref{eq:hypothesis_H} one can show that $d_1>z_H$ (see \cite{Garriz-Leculier-Mirrahimi}-Lemma 4.2) and that
\[
x(1+H'(x))-2H(x)<0\text{ for all }0<x<d_1\quad\text{and}\quad x(1+H'(x))-2H(x)>0\text{ for all }x>d_1,
\]
implying
\[
\big[\partial_x\left[x(1+H'(x))-2H(x)\right]\big]_{x=d_1}=1-H'(d_1)+d_1H''(d_1)\geq 0,
\]
which in turn implies
\[
H''(d_1)\geq -\frac{1-H'(d_1)}{d_1}.
\]
Therefore,
\[
\phi(d_1)\geq (1-H'(d_1))\cdot \left(1+H'(d_1)-\frac{2H(d_1)}{d_1}\right)=0
\]
since, again, $d_1$ is a root of $x(1+H'(x))-2H(x)$.

Finally, we compute $\phi'(x)=2H(x)H'''(x)>0$ for all $x>z_H$, and in particular for all $x\geq d_1$ since $d_1>z_H$. Combining  $\phi(d_1)\geq 0$ and $\phi'(x)>0$, for all $x\geq d_1$, implies that $\phi(x)>0$ for all $x>d_1$, which in turn yields~\eqref{eq:monotone_increasing}.

Properties~\eqref{eq:monotone_increasing} and~\eqref{eq:mu_greater_d} lead to
\[
\mu>\frac{2H(d_1)}{(1+H'(d_1))\cdot(1-H'(d_1))}=\frac{d_1}{1-H'(d_1)}=\mu_1,
\]
which is in contradiction with our hypothesis $\mu\leq \mu_1$.

\end{proof}

\begin{lem}\label{lem:delta_sigma_lambda}
	Let conditions of Theorem~\ref{thm:limit_epsilon_non_concave} hold and that $\mu\leq\mu_1$, $\bar{z}(0)\in\left(-1/\sqrt{g},\mu\right)$, $T_m<\infty$.  Then there exist positive constants $\delta, \sigma, \lambda$ such that
	\[
	\begin{aligned}
		&z_0\leq \z(t)\leq \mu-\delta\quad\text{ for all }\quad t\in[0,T_m],\\
		&F(t,z)<-\lambda\qquad\quad\text{ for all }\quad t\in[0,T_m],\ z\in (-\infty,\z(t)-\sigma)\quad \text{and}\\
		&\partial_z F(t,z)>0\qquad\quad\text{ for all }\quad t\in[0,T_m],\ z\in [\z(t)-\sigma, \z(t)].
	\end{aligned}
	\]
\end{lem}
\begin{proof}
	Since $\z(t)$ is monotone increasing, we begin by taking $\delta:=\mu-\z(T_m)$. Next, we notice that the monotonicity of $\z(t)$ implies
	\[
	2g\delta=2g(\mu-\z(T_m))\leq \partial_{z}F(t,\z(t))\leq 2g(\mu-\z(0))
	\]
	and thus, since $\partial_{zz}^2F$ is bounded, there must exist a positive constant $\sigma$ such that
	\[
	\partial_{z}F(t,z)>0\text{ for all }t\in[0,T_m], \ z\in [\z(t)-\sigma,\z(t)].
	\]
	Define now $\lambda_1$ by
	\[
	-\lambda_1=\sup\limits_{t\in[0,T_m],\ z\leq \z(t)-\sigma} F(t,z).
	\]
	From Lemma~\ref{lem:F_negative} we find that $-\lambda_1<0$. Choosing $\lambda<\lambda_1$ the claim is satisfied.
\end{proof}

This lemma and the fact that $\partial_{zz}^2 F(t,z)<-2g$ for all $z\geq \z(t)$, thanks to Lemma \ref{lem:prop_F}, allow us to conclude the existence of bounded intervals
\[
\mathcal{U}(t):=(\omega_1(t), \mu]
\]
satisfying
\[
	\begin{aligned}
 	[\z(t),\mu]\subset \mathcal{U}(t)\quad&\text{for all }t\in[0,T_m],\\
 	\partial_{zz}F(t,z)<0\quad&\text{for all }t\in[0,T_m],\ z\in\mathcal{U}(t).
	\end{aligned}
\]
  Note however that for now, for each $t$ the set $\mathcal{U}(t)$ is not unique. On the other hand, we expect to have that {$\z(t)\to \mu$ as $t\to\infty$.}

In addition, there exists a constant  $\tilde\sigma>0$ such that
\[
-2g - \tau\sup\limits_{z\in\R} H''(z)\leq \partial_{zz}^2F(t,z) =-2g+\tau H''(z-\z(t))<-g\quad\text{for all }t>0,\ z>\z(t)-\tilde\sigma,
\]
implying that $\mathcal{U}(t)$ can be chosen for each $t$ in a way such that
\[
-K_2\leq \partial_{zz}F(t,z)\leq -K_1<0\quad \text{ for all }\quad t\in[0,T_m],\ z\in \mathcal{U}(t),
\]
for a couple of positive values $K_1, K_2$ independent of $\z(t)$. Therefore, the concavity of $F$ in the set
\begin{equation}\label{eq:def_U}
\mathcal{U}:=\{(t,z)\in\R^2\ :\ t\in[0, T_m],\ z\in\mathcal{U}(t)\}
\end{equation}
is uniform. Let us then define, with $\sigma$ coming from Lemma~\ref{lem:delta_sigma_lambda},
\[
\sigma^*:=\min(\sigma, \tilde\sigma)
\]
and define
\[
\omega_1(t):= \z(t)-\f{3\sigma^\ast}{4}\text{ for all }t\in[0, T_m].
\]
We then notice that 
\begin{equation}\label{eq:conditions_U}
	\begin{aligned}
		[\z(t),\mu]\subset \mathcal{U}(t)\quad&\text{ for all }\quad t\in[0,T_m],\\
		\partial_{z}F(t,\mu)<0<\partial_{z}F(t,\omega_1(t))\quad &\text{ for all }\quad t\in[0,T_m],\\
		|\partial_{zzz}^3F(t,z)|\leq K\quad &\text{ for all }\quad t\in[0,T_m],\ z\in \mathcal{U}(t)\\
		-K_2\leq \partial_{zz}F(t,z)\leq -K_1<0\quad &\text{ for all }\quad t\in[0,T_m],\ z\in \mathcal{U}(t)\quad\text{and}\\
		 F(t,z)>F(t,y) \quad&\text{ for all }\quad t\in[0,T_m],\ z\in \mathcal{U}(t), y<\omega_1(t).
	\end{aligned}
\end{equation}
for three positive values $K, K_1, K_2$ independent of $\z(t)$ or $t$. These properties play a similar role to the hypothesis~\eqref{hyp1}--\eqref{asRD3} that are satisfied in $\Omega_0$ instead of $\mathcal{U}(t)$. These sets also have the advantage of satisfying $\mathcal{U}(t_2)\subset \mathcal{U}(t_1)$ for all $t_1<t_2$.

However, all the conditions of  hypothesis~\eqref{hyp2} are not satisfied, since by Lemma~\ref{lem:prop_F} there are points $z$ to the right of $\mu$ such that $F(t,z)>F(t,\z(t))=0$. This hypothesis~\eqref{hyp2} is used in the following points of Section~\ref{sec:Dyn_Prog}: in Step 2 and Step 6 of Section~\ref{sec:Dyn_Prog_1} and in the proof inside Section~\ref{sec:Dyn_Prog_2} that we omitted for the sake of brevity. In what follows we will show how to deal with these issues in order to prove similar results as in Section~\ref{sec:Dyn_Prog}.

In the proof of Section~\ref{sec:Dyn_Prog_2}, the hypothesis~\eqref{hyp2} is used to prove that $\z(t)\not\in\partial\Omega_0$, see~\cite{Mirrahimi-Roquejoffre - 2016} for more details. Here, the analogous property (i.e., $\z(t)\not\in\partial\mathcal{U}(t)$) holds thanks to the inequality $\omega_1(t)<\z(t)<\mu$, which is true by the construction of $\omega_1(t)$ and equation~\eqref{eq:z_is_Lipschitz}.

In Step 6 of Section~\ref{sec:Dyn_Prog_1} the hypothesis is used to prove that the maximum value $u=0$ is always attained inside $\Omega_0$ (again, in this case $\Omega_0$ should be changed by $\mathcal{U}(t)$). By the definition of $T_m$, we already know that, for all $t\in [0,T_m)$ the maximum value of $u$ is only attained at the point $\z(t)\in \mathcal{U}(t)$.  It remains to prove that the value $u(T_m,\cdot)=0$ is never attained outside $\mathcal{U}(T_m)$. The following Lemmata~\ref{lem:zeroes_1} and~\ref{lem:zeroes_2} will provide this.

But before that, we recall that the value of the viscosity solution $u$ of the Hamilton-Jacobi equation~\eqref{eq:v_ham_jac} at point $(t,x)$ with $t\leq T_m$ is given  by formula~\eqref{eq:dyn_prog}. Note also that, given $(t,x)\in \R^+\times\R$, by Lemma~\ref{lem:existence_curve_gamma} there exists an optimal curve $\gamma_x$  satisfying
\[
u(t,x)=f_t(\gamma_x),
\]
and, given the value of $u(t_0, \gamma_x(t_0))$, then we can also work with
$$
u(t,x) = \sup_{(\gamma(s),s)\in\R^d\times[t_0,t], \gamma(t)=x, \gamma(t_0)=\gamma_x(t_0)}\Big\{f_t(\gamma),\text{ where }\gamma\in W^{1,2}([t_0,t] : \R^d)\Big\},
$$
with
$$
f_t(\gamma)=u(t_0,\gamma_x(t_0))+\int_{t_0}^t\Big(-\frac{|\dot\gamma(s)|^2}4+F(s, \gamma(s))\Big)ds.
$$

\begin{lem}\label{lem:zeroes_1}
 Let conditions of Theorem~\ref{thm:limit_epsilon_non_concave} hold that and $\mu\leq\mu_1$, $\bar{z}(0)\in\left(-1/\sqrt{g},\mu\right)$, $T_m<\infty$. For every $z<\omega_1(T_m)$ we have that $u(T_m, z)<0$.
\end{lem}
\begin{proof} 
Let us argue by a contradiction argument using the representation formula. Let us suppose that $(T_m, x)$ is a point where $u(T_m, x)=0$, with $x<\omega_1(T_m)$, and let $\gamma_{x}$ be any $W^{1,2}$ curve (which may not be unique) such that
\[
u(T_m, x)=u_0(\gamma_{x}(0))+\int_0^{T_m}\Big(-\frac{|\dot\gamma_{x}(s)|^2}4+F(s, \gamma_{x}(s))\Big)ds.
\]

Since $u$ and $\gamma_{x}$ are continuous, it means that there must exist a time $t_0$ such that $\gamma_{x} (t)<\omega_1(t)$  for all $t\in [t_0,T_m] $ and $u(t_0,\gamma_{x}(t_0))<0$. Moreover, since $F(t,\bar{z}(t))=0$, by the definition of $\mathcal{U}(t)$, we obtain that for all $s\in [t_0, t_1]$,
\[
F(s,\gamma_{x}(s))<0.
\]
Then, using the representation formula  
\[
u(T_m, x)=u(t_0,\gamma_{x}(t_0))+\int_{t_0}^{T_m}\Big(-\frac{|\dot\gamma_{x_1}(s)|^2}4+F(s,\gamma_{x}(s))\Big)ds,
\]
we obtain that $u(T_m, x)<0$, a contradiction.
\end{proof}

\begin{lem} \label{lem:zeroes_2}
 Let conditions of Theorem~\ref{thm:limit_epsilon_non_concave} hold and that $\mu\leq\mu_1$, $\bar{z}(0)\in\left(-1/\sqrt{g},\mu\right)$, $T_m<\infty$. For every $z_1>\mu$ and $t\in [0,T_m]$, we have that $u(t, z_1)<u(t, \mu)$.
\end{lem}
\begin{proof}
Let  $z_1>\mu$ and $t\in [0,T_m]$. We divide the set of all possibles curves $\gamma_{z_1}$ such that $\gamma_{z_1}(t)=z_1$   into two subsets: those who are always to the right of the straight line $\bar{\gamma_\mu}:=[0,t]\times\{z=\mu\}$ and those who cross it.

Since straight lines are local maximizers of the functional
\[
\int_0^t-\frac{|\dot{\overline\gamma}(s)|^2}4ds,
\]
from the definition of $f$, the fact that $u_0(z)$ is monotone decreasing for all $z>\z(0)$ and the fact that $F(t,z)$ is monotone decreasing for all $t\in[0, T_m),\ z>\mu$, it is clear that for every curve $\gamma_{z_1}$ that do not cross the line $\bar{\gamma_\mu}$ we have that
$$
f_t(\gamma_{z_1})<f_t(\bar{\gamma_\mu})\leq u(t,\mu).
$$

We look now to the curves $\gamma_{z_1}$ that do cross this straight line. Let
$$
t_\mu:=\sup\{t\in(0,t) : \gamma_{z_1}(t)=\mu\}
$$
and define $\widetilde \gamma_\mu$ as
$$
\begin{cases}
\widetilde \gamma_\mu(s)=\gamma_x(s)& \text{for $s\in [0,t_\mu]$},\\
\widetilde \gamma_\mu(s)=\mu& \text{for $s\in (t_\mu,t]$}.
\end{cases}
$$
This curve $\widetilde \gamma_\mu$ is an acceptable curve for the functional $f$, so again, since $F(t,z)$ is monotone decreasing for all $t\in[0, T_m),\ z>\mu$, we conclude that
$$
f_t(\gamma_{z_1})<f_t(\widetilde \gamma_\mu)\leq u(t,\mu),
$$
which proves the claim.
\end{proof}
We deduce that the only possible   zeroes of $u$ have to be in $\mathcal{U}(T_m)$.

Finally, in Step 2 of Section~\ref{sec:Dyn_Prog_1} the hypothesis~\eqref{hyp2} is used to prove that optimal trajectories $\gamma(\cdot)$ with ending points in $\Omega_0$ have always been inside $\Omega_0$. Since hypothesis~\eqref{hyp2} is satisfied to the left of $\omega_1(t)$, it is enough to prove the following lemma.

\begin{lem}\label{lem:trayectories_less_mu}
	 Let conditions of Theorem~\ref{thm:limit_epsilon_non_concave} hold, $t_0\in(0,T_m]$, $z_1\in\mathcal{U}(t_0)$ and $\gamma_{z_1}(\cdot)$ be any curve  such that $\gamma_{z_1}(t_0)=z_1$ and $u(t_0, z_1)=f_t(\gamma_{z_1})$. Then $\gamma_{z_1}(t)\leq\mu$ for all $t\in[0,t_0]$.
\end{lem}
\begin{proof}
	Suppose that there exist $t'\leq t_0$ such that $\gamma_{z_1}(t)>\mu$. Then due to the continuity of $\gamma_{z_1}$ there must exist $t_1< t_2$ such that $[t_1, t_2]\subset[0, t_0]$ and $\gamma_{z_1}(t)>\mu$ for all $t\in(t_1, t_2)$, with $\gamma_{z_1}(t_2)=\mu$, and $\gamma_{z_1}(t_1)=\mu$ or $t_1=0$. Then we can construct a curve $\tilde\gamma_{z_1}(t)$ defined as
	\[
	\tilde\gamma_{z_1}(t):=\begin{cases}
		\gamma_{z_1}(t),  &\text{ for }t\in[0, t_1)\cup(t_2, t_0],\\
		\mu, &\text{ for }t\in[t_1, t_2],
	\end{cases}
	\]
	and, since $F(t,z)$ is decreasing for all $z\geq \mu$ and $u_0(z)$ decreases for all $z>\z(0)$, then it is clear that $f_t(\tilde\gamma_{z_1})>f_t(\gamma_x)$, contradicting the choice of $\gamma_{z_1}$.
\end{proof}

Therefore we have  results that make up for the lack of hypothesis~\eqref{hyp2} in the Step 2 and Step 6 of Section~\ref{sec:Dyn_Prog_1} and in Step 1 of Section~\ref{sec:Dyn_Prog_2}.

We are almost ready to adapt the ideas from Section~\ref{sec:Dyn_Prog} to study $u$ in the set $\mathcal{U}$ defined in~\eqref{eq:def_U}.  However, we cannot directly use the set $\mathcal{U}$ as the concavity set since $\mathcal{U}$ is not of the form $[0,T]\times \Omega_0$. Such a set is convex and thus every straight line joining two points inside it will be contained in $[0,T]\times \Omega_0$. These straight paths are used in Step 2 of Section~\ref{sec:Dyn_Prog_1}.
To overcome this difficulty we will split $[0,T_m]$ into the union of a finite number of time intervals. We consider a finite sequence of times, to be chosen later, $\{t_i\}_{i=1}^k$, such that $0= t_k<\dots<t_1=T_m$, and we define
\[
\mathcal{R}_i=[t_{i+1}, t_{i}]\times \overline{\mathcal{U}(t_{i})}.
\] 
Since $\omega_1(t)$ is monotone increasing and smooth we have that 
$$
\mathcal{U}(t_{i+1})\subset \mathcal{U}(t_{i}),
$$
and clearly, for any possible sequence $\{t_i\}_{i=1}^k$, we have that
\[
\bigcup\limits_{i=1}^{k-1}\mathcal{R}_i\subset \mathcal{U}.
\]
A possible sketch of such construction can be found in Figure~\ref{fig:rectangles}.
\begin{figure}[H]
	\centering
	\includegraphics[scale = 0.53]{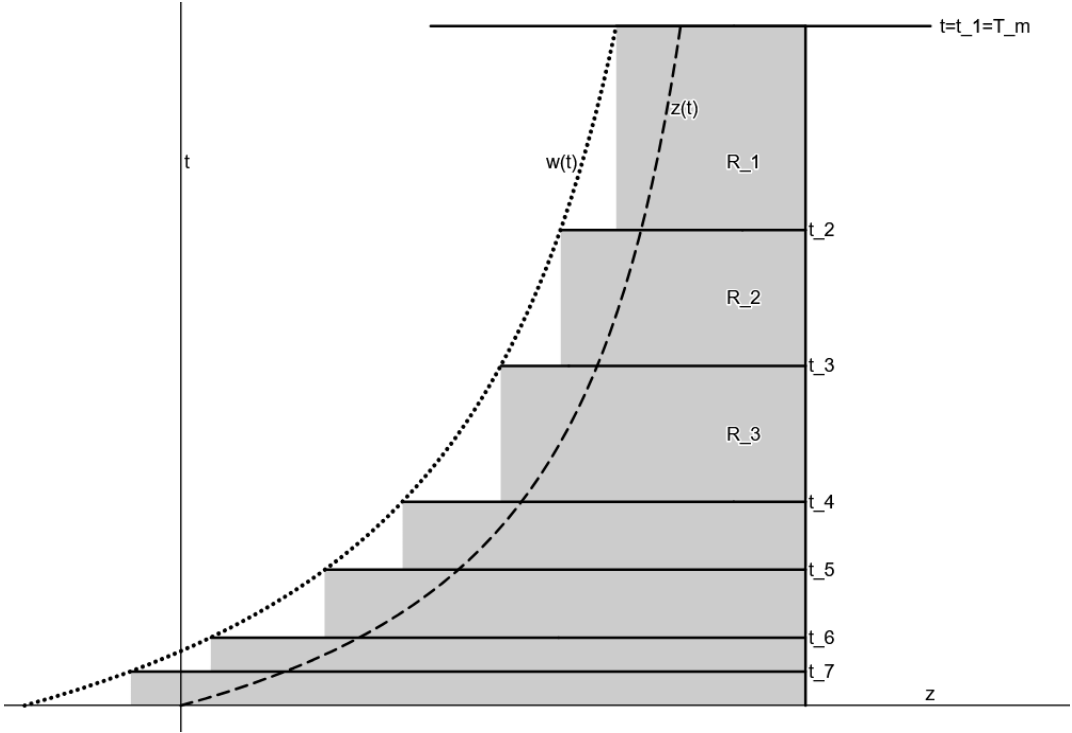}
	\caption{A collection of rectangles $\{\mathcal{R}_i\}_{i=1}^7$, in grey, that contain the curve $\z(t)$, in dashed-black. Each horizontal black line is a set $\mathcal{U}(t_i)$, for $i=1,...,7$. In dotted-black, the curve $\omega_1(t)$ and the vertical line corresponds to $\{ z=\mu \}$.}
	\label{fig:rectangles}
\end{figure}

However, the image of the curves $\omega_1(t)$ and $\z(t)$ may be very close to each other (if, for example, $z'(t)$ approaches 0), so it has to be justified why we can choose such a finite sequence of rectangles. First, we prove that the distance between the images of the curves $\omega_1(t)$ and $\z(t)$, understood as geometrical objects in the space $\mathbb{R}^2$ endorsed with the euclidean distance, is finite and strictly positive.

\begin{lem}\label{lem:distance_curves}
	Let $T_m<\infty$. There exists a positive constant $\kappa>0$ such that
	\[
	\kappa:=\inf\limits_{t,s \in [0, T_m]} \text{dist}\big((t,\z(t)), (s,\omega_1(s))\big) = \inf\limits_{t,s \in [0, T_m]} \sqrt{(t-s)^2+(\z(t)-\omega_1(s))^2}>0.
	\]
\end{lem}
\begin{proof}
	Since $\omega_1(t)<\z(t)<\mu$ and both curves are monotone increasing it is clear that both curves are inside the compact set $[0,T_m]\times[\omega_1(0), \mu]$ and thus $\kappa<\infty$. In order to prove that $\kappa>0$ we argue by contradiction. So suppose that $\kappa=0$. Then there must exists a point $(t_0, \z(t_0))$ and a sequence $\{(t_n, \omega_1(t_n))\}$ such that
	\[
	\sqrt{(t_0-t_n)^2+(\z(t_0)-\omega_1(t_n))^2}\to 0\text{ as }n\to\infty.
	\]
	This implies that $t_n\to t_0$ and $\omega_1(t_n)\to \z(t_0)$. However, since $\omega_1(t)$ is a continuous curve, this implies that $\omega_1(t_n)\to\omega_1(t_0)$ and therefore $\omega_1(t_0)=\z(t_0)$. But this is a contradiction with the definition of $\omega_1(t)$, since by construction $\z(t)-\omega_1(t)>\frac{\sigma^*}{2}$ for all $t\in[0,T_m]$.
\end{proof}

\begin{lem}
\label{lem:rec}
	Let $T_m<\infty$. There exists a finite collection of rectangles $\{\mathcal{R}_i\}_{i=1}^n$ satisfying the following properties:  there exists a decreasing sequence $(t_i)_{i=1}^{n+1}$, with $t_1=T_m$ and $t_{n+1}=0$, such that
	\[
	\mathcal{R}_i=[t_{i+1}, t_i]\times \overline{\mathcal{U}(t_i)}, 
	\]
	\[
	(t,\z(t))\in \cup_{i=1}^n\mathcal  {R}_i^\circ, \quad \text{for all $t\in[0,T_m]$}.
	\] 
\end{lem}
\begin{proof}
	We begin by taking $t_1=T_m$ and focus on the point $(t_1,\omega_1(t_1))$. Since the two curves that we are studying are continuous and monotone increasing, for every value $\omega_1(s)\in[\z(0), \omega_1(t_1)]$ there   exists a time $t_{\omega_1(s)}<s$ such that $\z(t_{\omega_1(s)})=\omega_1(s)$. On the other hand, by Lemma~\ref{lem:distance_curves} the distance between the point $(t_1,\omega_1(t_1))$ and the point $(t_{\omega_1(t_1)},\z(t_{\omega_1(t_1)}))$ is no less than $\kappa$. In particular, $(t,\z(t))\not\in B_\frac{\kappa}{2}\big((t_1,\omega_1(t_1))\big)$ for all $t\in[0,T_m]$, with $B_r(x)$ the ball with radius $r$ and centre $x$.

	Since $\z(t_1)>\omega_1(t_1)$ and $\z(t)$ is continuous, we have
	\[
	(t,\z(t))\in\left[t_1-\frac{\kappa}{2}, t_1\right]\times\mathcal{U}(t_1)\quad\text{for all}\quad t\in\left[t_1-\frac{\kappa}{2}, t_1\right],
	\]
	so by choosing $t_2=t_1-\frac{\kappa}{2}$ and $\mathcal{R}_1=[t_{2}, t_1]\times \overline{\mathcal{U}(t_i)}$ we have a first rectangle. Now we consider $\mathcal{U}(t_2)$, focus on the point $(t_2,\omega_1(t_2))$ and repeat the process, finding a $t_3=t_2-\frac{\kappa}{2}$ and a second rectangle. Note that the size of each interval is $\kappa/2$, strictly positive.
	
	We repeat this process until reaching a point $t_n> 0$ such that $\omega_1(t_n)<\z(0)$. Clearly, since $\kappa$ is defined as the distance between the images of the curves $\omega_1(t)$ and $\z(t)$,  $t_{i}-t_{i+1}=\frac{\kappa}{2}>0$ for all $i=1,...,n-1$, and   since $T_m$ is finite,  we reach such a $t_n$ in a finite number of steps. We conclude by choosing a last rectangle of the form $\mathcal{R}_n=[0, t_n]\times \overline{\mathcal{U}(t_n)}$.
\end{proof}

We highlight two things. First, the choice of the sequence $\{t_i\}_{i=1}^n$ satisfying the conditions of Lemma \ref{lem:rec} is not unique. Second, we are letting points inside $\mathcal{U}$ outside the rectangles $\mathcal{R}_i$, but as long as $\z(t)$ is inside the collection, we can proceed with the analysis.

Summarizing, we have divided most of the domain $\mathcal{U}$ into a sort of pyramid made of rectangles. We shall study the function $u$ at the points inside $\mathcal{U}(T_m)$ by applying the techniques from Section~\ref{sec:Dyn_Prog} in succession on each rectangle $\mathcal{R}_i$; let us describe how.

We begin by the last rectangle $\mathcal{R}_n$. By the properties of  the initial datum, equations~\eqref{eq:conditions_U} and by Lemma~\ref{lem:trayectories_less_mu}, we can repeat the usual argument of optimal trajectories based on formula~\eqref{eq:dyn_prog} in order to see that every optimal path $\gamma$ with ending point in the upper base of $\mathcal{R}_n$ is inside said rectangle for all times, it cannot take values outside it, not even in other points of $\mathcal{U}$ outside of $\mathcal{R}_n$. This property can indeed be proved using the fact that $F(t,z)$ is increasing with respect to $z$, for all $z<\z(t)$. If a trajectory leaves the rectangle, we can hence replace the part of the trajectory that leaves the domain by a straight line. Moreover, this path $\gamma$ is unique for each point in $\mathcal{R}_n$.   From the method displayed in Section~\ref{sec:Dyn_Prog} and with the help of Lemmata~\ref{lem:zeroes_1},~\ref{lem:zeroes_2} and~\ref{lem:trayectories_less_mu} we deduce that the solution $u$ at time $t_n$ is a strictly concave and smooth enough initial datum at the beginning of the next rectangle $\mathcal{R}_{n-1}$, satisfying the conditions for the initial datum imposed in these notes. If we iterate now on each rectangle of the collection $\{\mathcal{R}_i\}$ we can reach the set $\mathcal{U}(T_m)$, and, in the process, every point inside the collection $\{\mathcal{R}_i\}$. Let us define then
\[
\mathcal{C}:=\bigcup_{i=1}^n \mathcal{R}_i.
\]

\begin{lem} \label{lem:regularity_in_the_good_set}
	 Let conditions of Theorem~\ref{thm:limit_epsilon_non_concave} hold and that $\z(0)<\mu\leq \mu_1$, $T_m<\infty$. The solution $u$ of equation~\eqref{eq:v_ham_jac} is $C^1$ in time and $C^2$ in space in the set $\mathcal{C}$. At each time $t\in[0,T_m]$, it attains its maximum value only at the point $\bar{z}(t)$, with $(t,\z(t))\in \mathcal{C}$ and  for all $(t,z)\in \mathcal{C}$, we have
	\begin{equation}\label{eq:bounds_concavity_z(t)}
		-S_c\leq\partial_{zz}u(t, z)\leq -2\lambda,
	\end{equation}
	where $\lambda$ is a strictly positive constant depending only on the second derivative of the initial datum and of the fitness term $F$, and $S_c$ comes from Proposition~\ref{prp:regularity}. As a consequence, we have that
	\begin{equation}\label{eq:derivative_z_max}
	\dot\z(t)=\frac{2g}{\left(-\partial_{zz}u(t, \z(t))\right)}(\mu-\z(t)),
	\end{equation}
	\[
	\mu-(\mu-\z(0))e^{-\frac{2g}{S_c}t}\leq \z(t)\leq \mu-(\mu-\z(0))e^{-\frac{g}{\lambda}t}.
	\]
\end{lem}
\begin{proof}
	The idea consists in applying the results of Section~\ref{sec:Dyn_Prog} on each rectangle $\mathcal{R}_i$ successively. We begin by the last one, $\mathcal{R}_n$, letting  $\mathcal R_n=[0,t_n]\times [z_n,\mu]$.  Notice that thanks to \eqref{eq:conditions_U}, $F$ is strictly concave in $\mathcal{R}_n$, $F(t,z)>F(t,y)$ for all   $(t,z)\in \mathcal{R}_n$ and $y<z_n$. Moreover, $F(t,\mu)>F(t,y)$, for all $t\in [0,t_n]$ and $y>\mu$,  thanks to Lemma \ref{lem:prop_F}. 
	We hence obtain that the solution $u$ is $C^1$ in time and $C^2$ in space in the set $\mathcal{R}_n$. Moreover, at time $t_n$, the function $u$ is an admissible initial datum for studying the problem in the next rectangle $\mathcal{R}_{n-1}$ and it satisfies, by 
	 Lemma \ref{lem:zeroes_2} and Step 4 of Section~\ref{sec:Dyn_Prog_1}, for all $(t,z)\in \mathcal{R}_n$,  
	\[
	 u(t,y)\leq u(t,\mu)\quad \text{for all $y\geq \mu$,}
	\]
	\[
		-S_c\leq\partial_{zz}u(t, z)\leq -2\lambda,\quad\text{where}\quad \lambda=\frac{1}{2}\min\left(C_1, \frac{\sqrt{K_1}}{\sqrt{2}}\right),
	\]
	with $C_1$ coming from Hypothesis~\eqref{eq:hypothesis_u_0} and $K_1$ coming from~\eqref{eq:conditions_U}. 	  Note that these bounds do not depend on $t$ or $\z(t)$.  Moreover, following similar arguments to Lemma \ref{lem:zeroes_2}, we can also prove that 
	\[
	u(t,y_1)\leq u(t,z_n), \quad \text{for all $y_1<z_n$.}
	\]
We can then repeat this process in the finite number of rectangles of the collection $\mathcal{C}$ up to time $T_m$. On the process, we obtain~\eqref{eq:bounds_concavity_z(t)}.
	
	To obtain the last result, we use the equivalent of equation~\eqref{can-eq} for $u$, that is~\eqref{eq:derivative_z_max}
	and we employ also the bound in~\eqref{eq:bounds_concavity_z(t)}.
\end{proof}

\begin{rem}
	Thanks to the regularity of $\z(t)$ and $\omega_1(t)$, given any point $(t,z)\in\text{int}(\mathcal{U})$, we can find a finite collection of rectangles $\{\mathcal{R}_i\}_{i=1}^n$ such that $(t,z)\in \mathcal{R}_j$ for some $j\in\{1,..,n\}$, perhaps by having to take much smaller time intervals and many more rectangles than the necessary to cover only the curve $\z(t)$. This means that we can exchange in Lemma~\ref{lem:regularity_in_the_good_set} the set $\mathcal{C}$ by the set $\text{int}(\mathcal{U})$, although it is not mandatory for the analysis of the dynamics of $u$.
\end{rem}

It is also relevant that the constant $\lambda$ in~\eqref{eq:bounds_concavity_z(t)} is always strictly positive even when the domain $\mathcal{U}(t)$ gets smaller and smaller, since the second derivative of $F$ is strictly negative in $\mathcal{U}$ by the properties~\eqref{eq:conditions_U}. 

{\bf The proof of Theorem~\ref{thm:limit_epsilon_non_concave}-(i).} Assume that $0\leq \mu\leq \mu_1$, $z_0\leq \mu$ and $\tau\leq \sqrt{2g}$. From Lemma \ref{lem:valrho} we obtain that for all $s\in [0,T_m]$, $\rho(s)=1-g\z(s)^2>1-g\mu^2>0$. We hence deduce thanks to  Lemmata~\ref{lem:T_m=infty},~\ref{lem:zeroes_1},~\ref{lem:zeroes_2},~\ref{lem:regularity_in_the_good_set} that $T_m=+\infty$. All the statements of the theorem concerning $u$ and $\z$ follows. The uniqueness of $(u,\z)$ can also be proved following similar arguments as in the proof of Theorem \ref{thm:section_Dyn_Prog_2}, applying the method successively to the rectangles $\mathcal R_n$.

{\bf The proof of Theorem~\ref{thm:limit_epsilon_non_concave}-(ii).} Assume that $0\leq \mu\leq \mu_1$, $z_0\leq \mu$ and $\tau> \sqrt{2g}$. Using  Lemmata~\ref{lem:T_m=infty},~\ref{lem:zeroes_1},~\ref{lem:zeroes_2},~\ref{lem:regularity_in_the_good_set}, either $T_m=+\infty$ or $\rho(t)=0$, for some $t\in [0,T_m]$. Notice however from \eqref{eq:z_is_Lipschitz} that, if $T_m=+\infty$ then, as $t\to +\infty$, $\z(t)\to \mu$. We next use Lemma  \ref{lem:valrho} to obtain that $\rho(t)=\max(0,1-g\z(t)^2)$. Since $1-g\mu^2<0$, we deduce that for some $t\in [0,T_m]$, $\rho(t)=0$, hence statement (ii).

\section{A special case in the monomorphic range. The case $\z(0)>\mu$}\label{sect:z_0>mu}

The situation where $\z(0)>\mu$, with $\mu\leq \mu_1$, is a special one because one has to consider two different possibilities depending on the value of $\z(0)$. Either $F(0,z)$ has one positivity set or two, and both cases can produce very different behaviours of the solution, see Figure~\ref{fig:z_0>mu}. We will refer to these cases as $F(0,z)$ being \textit{type one} or \textit{type two} respectively.

\begin{figure}[H]\centering
	\renewcommand\thesubfigure{\Alph{subfigure}}
	\subfloat[\label{A}][$\z(0)=1,77$. \textit{Type one} $F(0,z)$]{ \includegraphics[scale = 0.3]{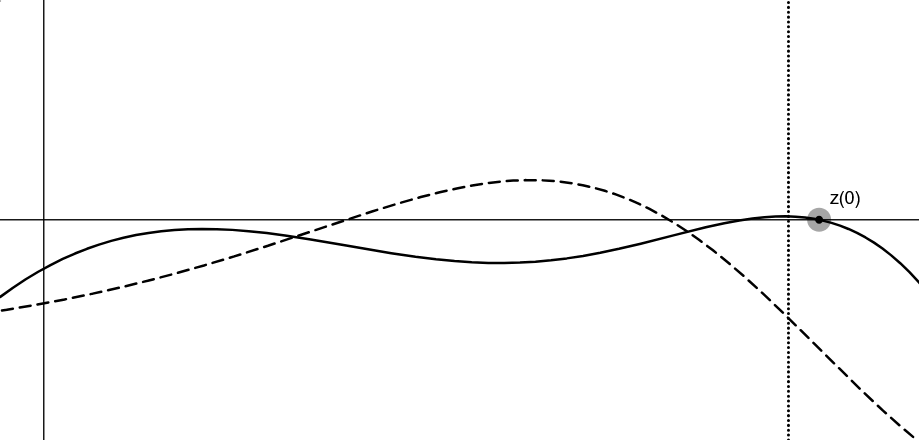}}
	\subfloat[\label{B}][$\z(0)=1,8$. \textit{Type two} $F(0,z)$]{ \includegraphics[scale = 0.3]{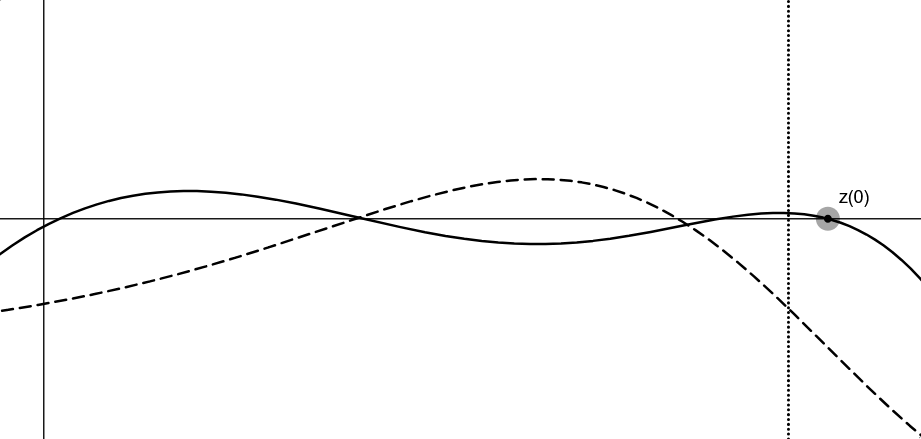}}\\[10pt]
	\caption{Two examples of $F(0,z)$. The second derivative $\partial_{zz}F(0,z)$ is in dashed-black, in order to appreciate concavity sets. As we can see, the value $\z(0)$ is to the right of the dotted vertical line representing $z=\mu$. The values of the parameters are $\tau=0,5$ and $\mu=1,7$. Note that in both cases $\mu<\z(0)<\mu_1$ and there is no remarkable difference between the functions $\partial_{zz}F(0,z)$.}
	\label{fig:z_0>mu}
\end{figure}

Repeating the arguments of Section~\ref{sect:z_0<mu} for fitness functions $F(0,z)$ of \textit{type one}, one can obtain a similar result to Theorem~\ref{thm:limit_epsilon_non_concave}, albeit this time $\z(t)$ decreases towards $\mu$. For the sake of brevity, we omit the details and just provide the result.

\begin{cor}Corollary of Sections~\ref{sect:general_considerations} and~\ref{sect:z_0<mu}. Assume~\eqref{eq:hypothesis_H},~\eqref{eq:hypothesis_u_0} and~\eqref{eq:hypothesis_rho}, and let $R(z)=1-gz^2$ and $0\leq \mu\leq \mu_1$. Assume also that $z_0>\mu$ is such that $F(0,z)$ is of type one. Then, the same conclusions of Theorem~\ref{thm:limit_epsilon_non_concave} hold.
\end{cor}

An open question is whether the same result holds true in the case where $F(0,z)$ is of \textit{type two}. The main obstacle for it is the fact that, if the initial datum is flat enough, then the second positivity set (the one to the left in Figure~\ref{fig:z_0>mu} (B)) can lead to the emergence of a second maximum point to the left of $\z(t)$. The solution may still be monomorphic in the sense that it has one maximum point for a.e. $t$, but it is not continuous monomorphic.

However, the positivity of the fitness term can be counter-balanced by the speed at which $\z(t)$ converges to $\mu$, since a very flat initial datum would also imply a very fast convergence, see formula~\eqref{eq:derivative_z_max}. If $\z(t)$ evolves fast enough, the left positivity set will disappear without having time to produce a new maximum point.

The study of this question requires a more detailed analysis and goes beyond the scope of this article.

\section{Beyond monomorphism. Open questions}\label{sect:beyond_monomorphism}

In the case where $R(z)=1-gz^2$, the monomorphic range $\mu\in[0,\mu_1]$ is well understood. However, when $\mu>\mu_1$ the situation is much more complex. Following the results from~\cite{Garriz-Leculier-Mirrahimi}, the natural next step in the study of this model is to focus on the so-called \textit{dimorphic range}; this is, the cases where $\mu\in(\mu_1, \mu_2)$ for a certain $\mu_2>\mu_1$. In~\cite{Garriz-Leculier-Mirrahimi} it was proven (under an extra hypothesis quite mild on $H$) that only in those cases there exists a stationary solution of~\eqref{eq:main} consisting of two points. Following this observation, let us describe, formally, the situation when $\mu>\mu_1$.

In the monomorphic regime the maximum point $\z(t)$ converges to $\mu$ and when doing so there is only one interval $I(t)$ where the fitness function $F$ is positive, and moreover $[\z(t), \mu]\subset I(t)$. In the limit, when $\z(t)=\mu$, we have that
\[
\lim\limits_{t\to\infty}\left(\max\limits_{z\in \R}F(t,z)\right)=\lim\limits_{t\to\infty}F(t,\mu)=0.
\]
At any other point, if $\mu<\mu_1$, the limit of the fitness function is negative; see Figure~\ref{fig:mu_less_mu1}.

\begin{figure}[H]
	\centering
     \includegraphics[scale = 0.75]{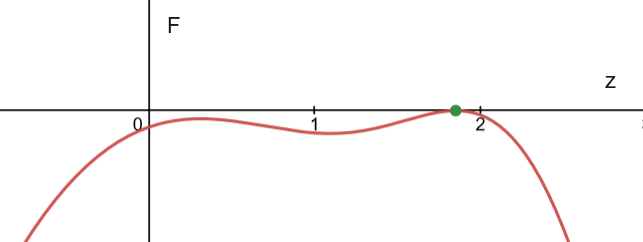}
		\caption{Fitness functions $F$ (red) with $\z(t)=\mu<\mu_1$. The point where $F=0$ is where $z=\mu$ (green).}
		\label{fig:mu_less_mu1}
\end{figure}

The limit case $\mu=\mu_1$ has the feature of presenting a second point $z_{\mu_1}$ such that
\[
\lim\limits_{t\to\infty}\left(\max\limits_{z\in \R}F(t,z)\right)=\lim\limits_{t\to\infty}F(t,\mu_1)=\lim\limits_{t\to\infty}F(t,z_{\mu_1})=0.
\]
and at any other point the limit in time of the fitness function is negative. However, it is still true that for all finite time $t$, the only positivity set of $F$ is the interval $I(t)$; see Figure~\ref{fig:mu_equal_mu1}.

\begin{figure}[H]
	\centering
     \includegraphics[scale = 0.75]{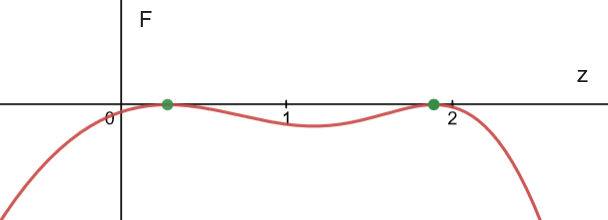}
		\caption{Fitness functions $F$ (red) with $\z(t)=\mu_1$. The points where $F=0$ are now $z=\mu_1$ and a second point $z=\mu_1-d_1$ (green), where $d_1$ comes from~\eqref{eq:d1}.}
		\label{fig:mu_equal_mu1}
\end{figure}

However, when $\mu>\mu_1$ and as $\z(t)$ converges to $\mu$, a second positivity set of $F$, call it $J(t)$, appears, and it is never empty as long as $\z(t)$ keeps moving towards $\mu$; see Figure~\ref{fig:mu_greater_mu1}. Let us define then
\[
J(t):=\{z\in (-\infty, \z(t)) : F(t,z)>0\}.
\]

\begin{figure}[H]
	\centering
     \includegraphics[scale = 0.75]{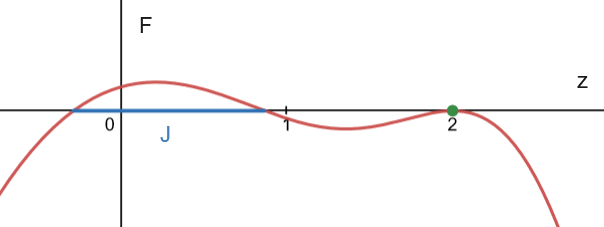}
		\caption{Fitness functions $F$ (red) with $\z(t)=\mu>\mu_1$. A set $J$ (blue) appears where $F>0$, far away from the point $z=\mu$ (green).}
		\label{fig:mu_greater_mu1}
\end{figure}

This, of course, provokes the appearance, sooner or later, of new maximum points of $u$ at a certain time $T$ inside the set $J(T)$, and when this happens several new questions arise:
\begin{enumerate}
\item Is the old maximum point $\z(t)$ still a maximum point after time $T$?
\item How many new maximum points appear? One, a discrete number of them or a whole interval?
\end{enumerate}

Next, if only one new maximum point appears, then we can imagine for instance the following possible behaviors.
\begin{enumerate}
\item The old maximum point persists and, for some time, there is dimorphism (the solution has two maximum points).
\item The old maximum point is no more a maximum point and the solution converges to a cyclic behaviour, where the new maximum point travels to the position of the old one and then the pattern repeats.
\item The new maximum travels for a while until a new maximum appears, and a sort of permanent oscillation between maximum points begin, while the solution is continuous monomorphic by intervals and converges to a stationary dimorphic solution.
\end{enumerate}

Stationary dimorphic solutions were studied in~\cite{Garriz-Leculier-Mirrahimi}, and it is known that they exist only in the range $\mu\in(\mu_1, \mu_2]$, for a certain $\mu_2>\mu_1$. Based on the simulations done for the study of this problem, it is the third option the one that seems most likely; see Figure~\ref{fig:di}.

\begin{figure}[H]\centering
        \renewcommand\thesubfigure{\Alph{subfigure}}
		\subfloat[\label{A}][Function $n_\varepsilon(t,z)$]{ \includegraphics[scale = 0.55]{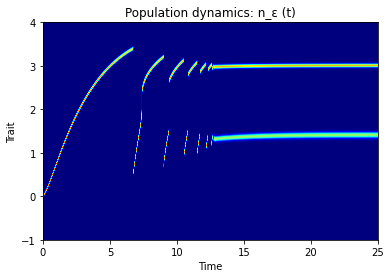}}
        \subfloat[\label{B}][Function $\rho(t)$]{ \includegraphics[scale = 0.55]{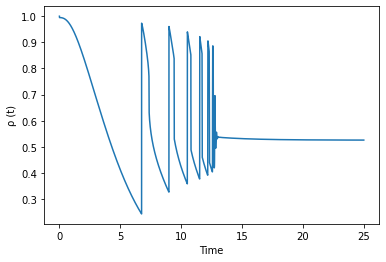}}\\[10pt]
		\caption{Solution $u_\varepsilon$ of~\eqref{eq:hopf-cole} in the dimorphic case $\mu=3.84$, with values $\tau=0.5$, $g=0.065$ and $\varepsilon=5\cdot 10^{-5}$. The oscillations between maximum points stop in a finite time due to the small parameter $\varepsilon>0$.}
        \label{fig:di}
	\end{figure}

\bigskip
\noindent \textbf{Acknowledgements}:
		The authors would  like to thank Cristobal Quininao and   Jean-Michel Roquejoffre for early discussions and computations within a project related to  Section \ref{sec:Dyn_Prog} of this article.
	This work has been supported by  the ANR project DEEV ANR-20-CE40-0011-01. The first author was supported by a post-doctoral fellowship of the ANR project DEEV ANR-20-CE40-0011-01.  The authors also thank the Chair ``Modélisation Mathématique et Biodiversité" of Veolia Environnement-Ecole Polytechnique-Museum National d’Histoire Naturelle-Fondation X and the European Union (ERC-AdG SINGER-101054787). Views and opinions expressed are however those
of the author(s) only and do not necessarily reflect those of the European Union or the European Research
Council. Neither the European Union nor the granting authority can be held responsible for them.

\end{document}